\documentclass[a4paper]{article}%{scrartcl}%{amsart}%
\usepackage[T1]{fontenc} % Change output font encoding
\usepackage{bbm}
\usepackage{lmodern}
\usepackage[utf8]{inputenc} % Change input font encoding 
\usepackage{amsmath}%
\usepackage{amsthm}%
\usepackage{amsfonts}%
\usepackage{amssymb}%
\usepackage{graphicx}
\usepackage[all]{xy}
\usepackage{mathtools}
\usepackage{multicol}
\usepackage{tikz}
\usetikzlibrary[arrows,decorations.pathmorphing]
\usepackage{color}
\usepackage{enumitem}

\usepackage{tocstyle}
\usetocstyle{noonewithdot}

\usepackage{url}
\usepackage{natbib}
\setcitestyle{numbers}

\usepackage{geometry}
\usepackage{setspace}
\theoremstyle{plain}

\newtheorem{theorem}{Theorem}
\newtheorem{proposition}{Proposition}
\newtheorem{lemma}[theorem]{Lemma}
\newtheorem{corollary}[theorem]{Corollary}

\theoremstyle{remark}
\newtheorem{remark}{Remark}
\newtheorem*{remark*}{Remark}
\newtheorem*{remarks*}{Remarks}
\newenvironment{remarks}{\begin{remarks*} \ \begin{itemize}}{\end{itemize}\end{remarks*}}

\theoremstyle{definition}
\newtheorem*{definition*}{Definition}
\newtheorem*{example*}{Example}% 
\newenvironment{examples*}{\begin{example*} \ \begin{itemize}}{\end{itemize}\end{example*}}
\newcommand{\calA}{\mathcal{A}}

\newcommand{\calD}{\mathcal{D}}

\newcommand{\calI}{\mathcal{I}}

\newcommand{\calN}{\mathcal{N}}

\newcommand{\C}{\mathbb{C}}

\newcommand{\Q}{\mathbb{Q}}
\newcommand{\R}{\mathbb{R}}
\newcommand{\Z}{\mathbb{Z}}

\DeclareMathOperator{\pr}{pr}

\newcommand{\tilda}{{\tilde{a}}}
\newcommand{\tildb}{{\tilde{b}}}
\newcommand{\tildc}{{\tilde{c}}}

\DeclareMathOperator{\Spec}{Spec} 

\def\clap#1{\hbox to 0pt{\hss#1\hss}}

\renewcommand{\P}{\mathbb{P}}
\DeclareMathOperator{\id}{id}
\DeclareMathOperator{\CH}{CH}

\DeclareMathOperator{\cube}{\square}%{\Box}%
\DeclareMathOperator{\cubeBar}{\overline{\cube}}%{\Box}%

\DeclareMathOperator{\Li}{Li}
\DeclareMathOperator{\dlog}{dlog}
\DeclareMathOperator{\spt}{spt}
\DeclareMathOperator{\Tot}{Tot}
\newcommand{\proj}{\mathit{proj}}

\newcommand{\alt}{\mathrm{alt}}

\newcommand{\sD}{{\calD}}

\newcommand{\pt}{\text{pt}}
\newcommand{\reg}{{\text{reg}}}
\newcommand{\Rm}{\mathbb{R}^{\raisebox{0.3ex}-}}
\newcommand{\Sh}{\mathrm{Sh}}      %Shuffles

\newcommand{\cl}{{\mathrm{cl}}}

\newcommand{\exclude}{\mathrm{omit}}

\newcommand{\scriptfrac}[2]{{\textstyle \frac{#1}{#2}}}
\newcommand{\ifShowAll}{\iftrue}
\begin{document}

%%%%%%%%%%%%%%%%%%%%%%%%%%%%%%%%%%%%%
%           TODO
%
%  - Decide whether to use american or british spelling
%
%
%%%%%%%%%%%%%%%%%%%%%%%%%%%%%%%%%%%%%%
\title{A commutative regulator map into \\ Deligne-Beilinson cohomology}
\author{Thomas Weißschuh}
%\address{Institut für Mathematik, Johannes Gutenberg Universität Mainz, Germany}
%\email{weisssth@uni-mainz.de}

\date{}
\maketitle

\begin{abstract}
In 1986, Bloch \cite{Bloch_AlgebraicCyclesAndBeilinsonsConjecture} gave an abstract definition of a map (regulator) from higher Chow groups to Deligne-Beilinson cohomology.
This map can be defined on the underlying complexes, % built out of currents (with the advantage of better functoriality), see King, El Zein/Zucker.
and Kerr, Lewis and M\"uller-Stach \cite{KLM_AbelJacobi} gave an explicit description of this map in terms of currents. 
Using a multiplicative version of the Deligne complex, we give a commutative version of the above map.

%Using a description of Deligne cohomology as a homotopy limit, we give a commutative version of the above map.
% 
%Their construction heavily relies on the Beilinson product $\bigcup_0$ on deligne cohomology. 
%Roughly they start with a map between weight 1 complexes and get the map on higher weights by repeatedly multiplying with the weight 1 map. 
%This note considers this construction with respect to various products, including a new product found by Levine, and compares the resulting regulators.
\end{abstract}

\tableofcontents

\newpage

\section*{Introduction}
For a algebraic manifold $U$ / $\C$ Beilinson \cite{BeilinsonHigherRegulatorsAndValuesOfLFunctions} defined rational motivic cohomology $H_M^i(U, \Q(n))$ and defined the regulator map into weight n Deligne-Beilinson cohomology of $U$, 
\[
    r_B : H^i_M(U,\Q(n)) \to H_\calD^i(U,\Q(n)).
\]
The terms occurring on either side of this map arise as the cohomology of certain complexes. 
It is suggested by Goncharov in \cite{Goncharov_Chow_polylogarithms_and_regulators} and \cite{Goncharov_Explicit_Regulator_Maps} that the regulator map should be induced by an explicitly defined map between these complexes\footnote{Bloch \cite{Bloch_AlgebraicCyclesAndBeilinsonsConjecture} proposed a cycle complex which computes motivic cohomology and even allows integral coefficients. He also constructed a map from the cohomology of this complex, the higher Chow groups, to the integral Deligne-Beilinson cohomology.}.
Kerr, Lewis and M\"uller-Stach \cite{KLM_AbelJacobi} gave such a map, using a variant of Bloch's higher Chow groups to compute motivic, and a cone-complex to compute Deligne-Beilinson cohomology. 
Although their map induces a graded-commutative map of graded-commutative algebras on cohomology, this is not true on the level of complexes. The aim of this note is to describe a map of partially defined graded-commutative dga's that on cohomology induces the same regulator as the one of \cite{KLM_AbelJacobi}.
We start by reviewing some complexes computing Deligne-Beilinson cohomology and motivic cohomology. 
On the motivic side, we work with Bloch's cycle complexes and in particular with their refinement $z^p_\R(U,\bullet)$ of higher Chow chains that have proper intersection with respect to some ''real faces''.
Then we give an abstract definition of such a regulator map that works for any family of complexes indexed by triples $(X,D,p)$ which satisfy a list of properties. 
Applying this construction to the complexes of higher Chow groups and Jannsen's complex of currents, we recover the regulator from \cite{KLM_AbelJacobi}.

To get a regulator between graded-commutative algebras, we replace Jannsen's complex by a complex $P_\calD$ that has a partially defined strictly graded-commutative product. 
The complex $P_\calD$ is related to Jannsen's complex $C_\calD$ by means of the evaluation map $ev$, which turns out to be a quasi-isomorphism after extending coefficients to $\Q$. 
The diagram formed by the two regulators and $ev$,
\[
    \vcenter{\xymatrix{z^p_\R(U,\bullet) \ar[r] \ar[rd]& P_D^{2p-\bullet} \ar[d]^{ev} \\ & C_\calD^{2p-\bullet}}}\quad ,
\]
although not commutative in general, commutes after passing to cohomology.

% Alternating Chow groups
Bloch's cycle complexes admit a partially defined intersection product, but this is graded-commutative only on cohomology. 
For rational coefficients however there exists subcomplexes $z^p_\R(U,\bullet)^{\alt}$ of alternating cycles that also compute motivic cohomology. These complexes are also endowed with a partially defined -- this time graded-commutative -- intersection product and we obtain the desired map as the restriction to these complexes. 

%For rational coefficients, motivic cohomology can also be computed by complexes of alternating cycles $z^p_\R(U,\bullet)^\alt$. These complexes are also endowed with a partially defined intersection product and we obtain the desired map as the restriction to these complexes. 

We give formulas for the regulator into $P_\calD$ for small simplicial degree.
To see a concrete example, we apply the regulator to a generalization of Totaro's cycle, which leads to dilogarithms. 

We also transfer the construction of the Abel-Jacobi map in \cite{KLM_AbelJacobi} to the regulator into $P_\calD$.

%We replace Jannsen's complex by a quasi-isomorphic (for $\Q$-coefficients) complex $P_\calD$ and obtain a regulator into a partially defined dga that has a strictly graded-commutative product. 
%We compare the two regulators by introducing an evaluation map $ev$ that connects $P_\calD$ with Jannsen's complex $C_\calD$, and show that the diagram formed by $ev$ and the two regulators,
%\[
%    \xymatrix{z^p_\R(U,\bullet) \ar[r] \ar[rd]& P_D^{2p-\bullet} \ar[d]^{ev} \\ & C_\calD^{2p-\bullet}},
%\]
%commutes in cohomology and $ev$ is a quasi-isomorphism there for (at least) $\Q \subset A$. 
%We give formulas for the regulator into $P_\calD$ for small simplicial degree.
%To see a concrete example, we apply the regulator to a generalization of Totaro's cycle, which leads to polylogarithms. 
%
%We also transfer the construction of the Abel-Jacobi map in \cite{KLM_AbelJacobi} to the regulator into $P_\calD$. 

The author thanks Stefan Müller-Stach and Marc Levine for their patience, motivation and support while writing this text. 

\textbf{Convention.} We work in the (complex) analytic case, i.e., all spaces / sheaves are with respect to the analytic topology. 
%; in particular, $\Omega$ denotes the sheaf of complex valued holomorphic differential forms. 

\section{Deligne-Beilinson cohomology}
\label{sec:deligne_cohomology}
Let $A \subset \R$ be a subring of coefficients and $A(p)= (2\pi i)^p A$ its $p$-th Tate twist. Consider a pair $(X,D)$ consisting of a compact manifold $X$ together with a normal crossing divisor $D$ on it. One may think of such a pair as a good compactification of the manifold $U = X \setminus D$. If $j$ denotes the inclusion $U \to X$, then the weight $p$ Deligne-Beilinson cohomology of $U$ is defined \cite{EsnaultViehweg} as the hypercohomology of the complex of sheaves
\begin{align*} 
    \Tot \Big( Rj_* A(p) \oplus F^p \Omega_X^\bullet(\log D) \xrightarrow{-\delta + \iota} \Omega_X^\bullet(\log D) \Big)   
\end{align*}
with the $F^p$ denoting the $p$-th part of the Hodge filtration and $\Omega_X(\log D)$ the shaf of complex valued holomorphic differential forms with at most logarithmic poles along $D$. The maps $\delta$ and $\iota$ are the inclusions. 

One can take this description as a starting point to get more "smooth" definitions of Deligne-Beilinson cohomology. 
Write $\calI(X,D,A(p))$ for the complex of relative integral currents and $\calD(X,\log D)$ for the complex of logarithmic currents (see the appendix \ref{appendix:currents}). 
They have the advantage of being complexes of abelian groups (instead of sheaves) with the same cohomology groups as $Rj_* A(p)$ resp. $\Omega_X(\log D)$ and furthermore are covariant functorial with respect to proper morphisms. 
For a pair $(X,D)$ as above, define
\begin{align} \label{eq:Deligne_Tot_complex_quasiProj}
    C_\calD(X, D, A(p)) := \Tot \left( \calI(X, D,A(p)) \oplus F^p \calD(X,\log D) \xrightarrow{ - \delta + \iota} \calD(X,\log D) \right),
\end{align}
where again $\delta$ and $\iota$ denote the canonical inclusions. 

The Deligne-Beilinson cohomology of a quasi-projective algebraic manifold $U$ over $\C$ with good compactification $X \supset U$ and normal crossing boundary divisor $D = X \setminus U$ is then 
\[
     H_\calD^n(U,A(p)) := H^n C_\calD^\bullet(X, D, A(p)).
\]
This definition is independent of the choice of the compactification \cite[1.13 a)]{Jannsen_DeligneHomology}.

Note that we are actually working with Deligne-Beilinson homology which, by \cite[1.15]{Jannsen_DeligneHomology}, is isomorphic to Deligne-Beilinson cohomology. The use of homological complexes to compute cohomology is analogous to the traditional use of higher Chow chains to compute motivic cohomology. In this way, we will get a map from cycles to cycle. Moreover the cohomological notation sometimes simplifies the indices (e.g., for the intersection product).

There is another description of Deligne-Beilinson cohomology that uses an interpretation of the above complex \eqref{eq:Deligne_Tot_complex_quasiProj} as a homotopy limit of the diagram $\calI \oplus F^p\calD \to \calD$. This homotopy limit can also be described\footnote{For example by applying the general procedure of Hinich/Schechtman \cite[Theorem 4.1]{HinichSchechtman} of Thom-Sullivan cochains to the above diagram (considered as a simplicial object). The quasi-isomorphism given therein specializes to the comparison-isomorphism in \ref{subsec:comparison_of_regulators}.} by paths in $\calD(X,\log D)$ connecting $\calI$ and $F^p\calD$. 
Formally, it is a subcomplex
\[ 
     P(X,D,A(p))  \subset \Omega_A(x) \otimes_A \calD(X, \log D)
\] 
where $\Omega_A(x)$ is the dga of $A$-valued polynomial differential forms on the $1$-simplex, i.e., the free graded-commutative differential graded $A$-algebra generated by a variable $x$ in degree $0$.
It is defined as
\[
    P^\bullet(X, D,A(p)) := \left\{ P \in \Omega(x) \otimes \calD(X, \log D) \text{ such that } 
    \genfrac{}{}{0pt}{}{ P_0 \in \calI(X, D, A(p)),} { P_1 \in F^p \calD(X, \log D) } \right\}
\]
where $P \mapsto P_\epsilon$ is the unique morphism that sends $x \mapsto \epsilon$, $dx \mapsto 0$ and is the identity on $\calD$. 
%It is freely generated by   
%\[
%    \left< 
%       w \otimes T \, \middle| \, 
%       \begin{aligned} w(0) T &\in \calI(X,D,A(p)),      \\ 
%                       w(1) T &\in F^p \calD(X, \log D) 
%       \end{aligned}   
%    \right> 
%\]
This complex naturally carries the structure of a graded commutative (partially defined) dg algebra, inherited from the tensor product of dg algebras $\Omega_A(x) \otimes_A \calD(X, \log D)$. It is only partially defined, because the intersection product on $\calD(X,\log D)$ is only partially defined. 

In case that $A \supset \Q$, this complex is quasi-isomorphic to $C_\calD$, see \ref{subsec:comparison_of_regulators}.

\section{Motivic cohomology} 
Motivic cohomology is a bigraded cohomology theory for algebraic varieties whose existence was conjectured by Grothendieck. 
There are several approaches to construct motivic cohomology. One very explicit variant uses the "higher Chow groups" introduced by Bloch \cite{BlochAlgebraicCycles}.

In their cubical incarnation they are defined as the homotopy groups of some cubical abelian groups $c^p(U, *)$ of so-called admissible chains of codimension $p$. An admissible chain is an algebraic cycle in $U \times \cube^n$, where $\cube^n = (\P^1 \setminus \{1\})^n$ is the algebraic $n$-cube, that meet all faces of the cube properly. Such a face is by definition the subvariety obtained by setting one (or more) coordinates to be $0$ or $\infty$. Degenerate cycles are defined as the pullbacks of admissible chains along one of the various projections $\pr_i$ which forget the $i$-th coordinate.

Instead of defining higher Chow groups as the homotopy groups of this cubical object, one can equivalently define them as the homology groups $\CH^p(U,n) = H_n z^p(U,-)$ of the associated normalized complex $z^p(U,-)$ which by definition is the complex 
%\[ 
%    z^p(U,n) = \left< Z \subset U \times \cube^n \text{ codimension p alg cycle that meet the faces properly } \right> / \text{ degenerate cycles }
%\]
\[ 
    z^p(U,n) = \left< \
               \parbox{0.4 \linewidth}{ $Z \subset U \times \cube^n$ codimension $p$ algebraic \\ cycle that meet all faces properly } 
               \right> _{\Big/  \text{ \normalsize degenerate cycles }}
\]

with differential given through intersection with the codimension $1$ faces, 
\[\partial = \sum_{i=1}^n (-1)^{i+1} \Big( (\partial_{i,0})^* - (\partial_{i,\infty})^* \Big).\]
Here $\partial_{i,\epsilon}$ denotes the map that insert $\epsilon \in \{0,\infty\}$ in position $i$.

Motivic cohomology of an algebraic manifold can actually be defined using higher Chow groups through (for Voevodsky's more modern definition and the equivalence with this one, see \cite[4.2.9]{Voevodsky_TriangulatedCategoriesOfMotivesOverAField})% \cite[p. 508]{Levine_Handbook})
\[
    H^{2p-n}_M(U,\Z(p)) = CH^p(U,\Z(n)). 
\]
Motivic cohomology can also be computed by a (sub)complex consisting of higher Chow chains which have proper intersection with some real subsets $\Rm = [-\infty , 0] \subset \cube$. 
Namely, denote by $c_\R^p(U,n)$ the free abelian group of admissible chains that meet (as real analytic chains) $U \times (T_{i_1} \cap \ldots \cap T_{i_j})$ properly (for $T_{i} = \{ z \in \cube^n : z_i \in \Rm \}$).
In the same manner as before define degenerate cycles and the quotient complex $z_\R^p(U,\bullet)$. 
This complex will be the natural domain for the regulator map. 
Kerr/Lewis \cite[Lemma 8.14]{KerrLewis_AbelJacobi_II} showed that the inclusion
\[
 z^p_\R(U,\bullet)_\Q \to z^p(U,\bullet)_\Q
\]
is a quasi-isomorphism.
%Its inclusion into $z^p(X,-)$ is known to be a quasi-isomorphism, as is shown e.g. by Kerr/Lewis \cite[Lemma 8.14]{KerrLewis_AbelJacobi_II}.
%
%Finally, there is a complex $z_\qfin^p(X,\bullet)$ of quasi-finite cycles, 
%\[
%   z_\qfin^p(X,n) = \left< Z \subset X \times (\A^1)^p \times \cube^n \mid \text{ The projection } Z \to X \times \cube^n \text{ is quasi-finite } \right>
%\]
%with differential induced by the same formula as above. Strictly speaking, it doesn't compute motivic cohomology for all possible $X$ (the associated Nisnevich sheaf however does) but for $X$ the spectrum of a field, it is quasi-isomorphic to $z^p(X,\bullet)$.

\paragraph*{Products and alternating cycles}

The higher Chow groups admit an intersection product that is induced by a partially defined product on the underlying chain complex \cite[p. 452f]{Levine_Handbook}. 
%Starting with two higher Chow chains, it forms their exterior power in $U \times \cube^n \times U \times \cube^m$ and pulls the result back to $U \times \cube^{n+m}$ along the map induced by the diagonal of $U$. 
Starting with two admissible chains, it forms their exterior power in $U \times \cube^n \times U \times \cube^m$ and pulls the result back to $U \times \cube^{n+m}$ along the map induced by the diagonal of $U$. This product is only partially defined, because the pullback may not exists. The intersection with a degenerate cycle is again a degenerate cycle and thus one gets a partially defined product on 
\[ \bigoplus_p z^p(U,\bullet) \] with unit the class of $U$ in $z^0(U,0)$, see also \cite[14,28]{Levine_Handbook}. 
 
This product is associative on the level of chains but is in general not graded-commutative. It however induces a graded-commutative product on cohomology. 
When passing to rational coefficients, one can force the product to be graded-commutative (on chains) by using alternating cycles \cite[29]{Levine_Handbook}.
%To define this consider the (unambiguous) semi-direct product $G_{n,p} := (\Z/2)^n \rtimes S_n \times S_p$. This group acts on $X \times (\A^1)^p \times \cube^n $ and hence on algebraic cycles therein as follows. Each factor $\Z/2$ acts by inverting the coordinate of the respective $\cube$, that is by $z \mapsto - \frac{1}{z}$. The group $S_n$ acts on $\cube^n$ by permutation of the coordinates (with appropriate signs), and $S_p$ acts on $(\A^1)^p$ by permuting coordinates (but without any signs). 
To define them, consider the action of the symmetric group $S_n$ on $U \times \cube^n$ by permutation of the coordinates (with appropriate signs).
This gives an action of $S_n$ on the set of algebraic cycles in $U \times \cube^n$ that preserves proper intersection with (real) faces and sends degenerate cycles to degenerate cycles. Thus it restricts to actions on $z^p(U,n)$ and $z_\R^p(U,n)$.
%This action of $G_{n,p}$ restricts to actions on $z^p(X,n)$ and $z_\qfin^p(X,n)$ respectively.
The alternating higher Chow chains are defined to be the subspace of $S_n$-invariant higher Chow chains
\[
    z^p(U,n)^{\alt} := z^p(U,n)^{S_n}.
\]
The differential $\partial$ restricts to these subspaces and makes them into a complex (with respect to the variable $n$). The projector associated to the $S_{n}$-action,
\[
    \alt : Z \mapsto \frac{1}{n!} \sum_{g \in S_{n}} g \cdot Z
\]
commutes with the differential, i.e., $\alt \circ \partial = \partial  \circ  \alt$. In particular, $\alt$ is a morphism of complexes from higher Chow chains to invariant chains with rational coefficients. 
%It is a quasi-isomorphism \cite[Lemma 29]{Levine_Handbook} when passing to rational coefficients on both sides.
Similarly, one defines $z^p_\R(U,n)^{\alt}$.

%The complexes of alternating cycles give another description of the rational motivic cohomology. 
The complexes of alternating cycles with rational coefficients inherit a product from the higher Chow complex: It is the usual product on Chow chains followed by the $\alt$ projection.
By construction, this product is strict graded-commutative with respect to $n$. The projector $\alt$ induces a map of partially defined differential graded algebras
%\[
%    \bigoplus \limits_{p,n} z^p(U,n)_\Q \to \bigoplus_{p,n} z^p(U,n)^{\alt}_\Q.
%\]
\[
    \bigoplus \limits_{p} z^p(U,\bullet)_\Q \to \bigoplus_{p} z^p(U,\bullet)^{\alt}_\Q.
\]
The situation is summarized by the 
\begin{proposition}
The diagram below formed by the alternating projection (horizontal) and inclusion (vertical) is a commutative diagram of partially defined dg algebras and each arrow is a quasi-isomorphism.
%\[
%\xymatrix{
%  \bigoplus \limits_{p} z^p(U,\bullet)_\Q \ar[r]            &  \bigoplus \limits_{p} z^p(U,\bullet)_\Q^{\alt}    \\
%  \bigoplus \limits_{p} z^p_\R(U,\bullet)_\Q \ar[u] \ar[r]  &  \bigoplus \limits_{p} z^p_\R(U,\bullet)_\Q^{\alt} \ar[u]
%}
%\]
\[
\xymatrix{
  \bigoplus \limits_{p,n} z^p(U,n)_\Q \ar[r]            &  \bigoplus \limits_{p,n} z^p(U,n)_\Q^{\alt}    \\
  \bigoplus \limits_{p,n} z^p_\R(U,n)_\Q \ar[u] \ar[r]  &  \bigoplus \limits_{p,n} z^p_\R(U,n)_\Q^{\alt} \ar[u]
}
\]
%The diagram formed by the alternating projection (horizontal) and inclusion (vertical) below is a commutative diagram of partially defined dg algebras
%\[
%\xymatrix{
%  \bigoplus \limits_{p,n} z^p(U,n)_\Q \ar[r]            &  \bigoplus \limits_{p,n} z^p(U,n)_\Q^{\Alt}    \\
%  \bigoplus \limits_{p,n} z^p_\R(U,n)_\Q \ar[u] \ar[r]  &  \bigoplus \limits_{p,n} z^p_\R(U,n)_\Q^{\Alt} \ar[u]
%}
%\]
\end{proposition}
\begin{proof}
Commutativity of the diagram is obvious. That the horizontal arrows are quasi-iso\-morphisms follows from general theory of (extended) cubical objects as in \cite[Proposition 1.6]{Levine_SmoothMotives} (there it is proved that the inclusion of alternating chains into chains is a quasi-isomorphism. But then $\alt$ is necessarily a quasi-inverse.). For the upper arrow this is also in \cite[Lemma 29]{Levine_Handbook}.
That the vertical arrows are quasi-isomorphisms, is content of the moving lemma. The left arrow is proven by Kerr/Lewis in \cite[Lemma 8.14]{KerrLewis_AbelJacobi_II} and for the right arrow it follows from the commutativity of the diagram and the quasi-isomorphicity of the other three arrows.

The statement that all maps are compatible with the products is clear for the inclusions. For the alternating projections, it is equivalent to $\alt(\alt(Z) \cdot \alt(Z')) = \alt(Z \cdot Z')$. This can be verified by hand.
\end{proof}
%In exactly the same way one proceeds for alternating quasi-finite cycles, obtaining an associative, graded-commutative (external and internal) product. The inclusion $z_\qfin(X,n)^{G_{n,p}} \to z(X \times (\A^1)^p,n)^{G_{n,0}}$ however is compatible with the external product if and only if $X$ is a point.

\section{An abstract regulator} \label{sec:abstract_regulator}
Assume that for each weight $p$ there is given a functorial collection of complexes of abelian groups $C_\calD^\bullet(X,D,\Z(p))$ for each pair of a projective algebraic manifold $X$ and a normal crossing divisor $D \subset X$. Think of these complexes as "Deligne complexes", whose cohomology calculate the Deligne-Beilinson cohomology of $U := X \setminus D$.
This collection should be covariant functorial for proper morphisms.
Assume that there exist cycle maps $\cl$, which to a codimension $p$ algebraic cycle in $U$ associate an element in $C_\calD^{2p}(X,D, \Z(p))$. 
We want to define higher cycle maps, or regulator maps, from higher Chow cycles to these Deligne complexes,
\[
     r_\calD : z^p(U,n) \to C_\calD^{2p-n}(X,D,\Z(p)).
\]
These maps should have some good properties. For example they should recover the cycle maps $\cl$ for $n=0$. The regulators for varying $n$ should be compatible, i.e., they should give rise to a morphism of complexes (suitably shifted). 
We are (in fact this was the starting point of this construction) interested in the multiplicative behavior of this map. Therefore we assume that on $\oplus_p C_\calD(X,D,\Z(p))$ there exist an exterior product $\boxtimes$ which is additive in both the weight $p$ and the degree, and which furthermore is unitary and associative in the sense of exterior products. Furthermore we assume that there exists a reasonable intersection product on these complexes. All these structures should be compatible in a sense which is made precise in the appendix \ref{appendix:requirements_on_D_complexes}.

%We also need the condition that for a projection $\pr : X \times Y \to X$ and a cycle $Z$ on $X$, one has $\pr_* (\cl(Z \times \pt)) = \pr_* \cl(Z)$.

To define the regulator map $r_\calD$ one may proceed as follows. 
First, compactify the cube to $\cubeBar = \P_1$ with marked point $\{1\}$ as boundary divisor. 
The cycle map gives for each higher chain $Z$ of codimension $p$ an element 
\[
     \cl( Z) \in C^{2p}(X \times \cubeBar^n, D \boxtimes \mathbbm{1}^n, \Z(p) ),
\]
where $\mathbbm{1}^n$ denotes the union of the various divisors $\{ z_i = 1 \}$ and $\boxtimes$ denotes the outer product of divisors (not to confuse with the exterior product in the complex $C_\calD$).
Then choose an element $R^1 \in C_\calD^1(\cubeBar, 1, \Z(1))$ and form the exterior products $R^n \in C^n_\calD(\cubeBar^n, \mathbbm{1}^n, \Z(n))$. 
Note that $R^0 \in C^0_\calD(\pt,\emptyset, \Z)$ is just the unit for the external product. 
%The product with the cycle associated to $X$ then results in 
%\[
%     \cl(X) \times R^n \in C^{n}(X \times \cubeBar^n, \Z(n)).
%\]
The value of a cycle $Z$ under the map $r_\calD$ is then defined as 
%\[
%    r_\calD(Z) := (2 \pi i)^{-n} (\pr_X)_* \Big( (\pr_{\cube^n}^{X \times \cube^n})^* (R^n) \cap \cl(Z) \Big) 
%\]
\[
    r_\calD(Z) :=  (\pr_X)_* \Big( (\cl(U) \boxtimes R^n) \cap \cl( Z ) \Big) 
\]
whenever the intersection product is defined. 
Remark that the push forward is always defined since the projection $(X \times \cubeBar^n, D \boxtimes \mathbbm{1}^n ) \to (X , D)$ has compact fibres, hence is proper. Thus the only obstruction to definedness is the non-existence of the intersection of $\cl(U) \boxtimes R^n$ with $\cl(Z)$.

To check well-definedness on higher Chow cycles, let $Z = \exclude_j^* \tilde Z$ be a degenerate cycle, obtained as the pullback of an admissible chain along the map that forgets the $j$th cube. It is to show that the regulator $r_\calD(Z)$ is zero. This follows from the projection formula if one knows that $(\exclude_j)_* (\cl(U) \boxtimes R^n)$ vanishes. But this is true for dimension reasons. Indeed, it is enough to show that $(\pr_\pt)_* R^1$ vanishes. But if it exists, this push forward would have dimension $-1$, hence must be zero. 
%But this is true since the fibre integral vanishes for dimension reasons.

Consequently, the regulator map descends to a partially defined map on higher Chow groups. This "abstract regulator map" serves as the prototype of a regulator. It becomes concrete after fixing a choice of the required data (the complexes $C_\calD(X,D,\Z(p))$ together with products and the element $R^1$).
%\footnote{For an argument why the abstract regulator should have this form, see appendix \ref{sec:why_the_regulator_should_have_this_form}.}
This is done in later sections. First, some compatibility conditions are exhibited.

\paragraph*{Compatibility with differential}
We now assume that $dR^1 = \cl((z))$ is the cycle associated to the divisor $(z) = 0 - \infty$. 
Then the regulator transforms the Bloch differential on higher Chow groups to the differential in $C_\calD$. 
To see this, we apply the differential to the regulator of a higher Chow chain $Z \subset U \times \cube^n$. 
For clarity, we abbreviate $\pr = \pr_X$ for any projection $X \times \cubeBar^? \to X$.% and similar $\pr^*$ for any pullback along some $X \times \cubeBar^? \to \cubeBar^?$. 
%Furthermore, we write $\cl(X)$ for $\cl(X \setminus D)$, i.e. we identify $\cl(X)$ with the image of $\cl(X \setminus D)$ under the map induced by $(X,0) \to (X,D)$.
We have to show that 
\[
      d\pr_* \big( (\cl(U) \boxtimes R^n) \cap \cl(Z) \big) = \pr_* \big( (\cl(U) \boxtimes  R^{n-1}) \cap \cl(\partial Z) \big).
\]
Since the differential commutes with direct images and using the product rule for $\boxtimes$ and $\cap$ (together with $d \cl(Z) = 0$) one gets
\begin{align*}
  d\pr_* \big( (\cl(U) \boxtimes R^n) \cap \cl(Z) \big) 
     &= \pr_* \big( (\cl(U) \boxtimes d R^n) \cap \cl(Z) \big) \\
     &= \pr_* \sum_{i=1}^n (-1)^{i+1} \left( \cl(U) \boxtimes R^{i-1} \boxtimes d R^1 \boxtimes R^{n-i} \right)  \cap \cl(Z).  \\
\intertext{%
Note that by asumption $d R^1 = \cl((z))$ is the difference of the classes of the two points $0$ and $\infty$.
Thus it can be interpreted as the push forward of these points under the inclusion into $\cubeBar$. 
Then using the projection formula one gets}
     &= \pr_* \sum_{i=1}^n (-1)^{i+1} (\partial_{i,0} - \partial_{i,\infty})_* \left( \cl(U) \boxtimes R^{n-1} \right) \cap \cl \left( Z \right)  \\
     &= \pr_* \sum_{i=1}^n (-1)^{i+1} ( \cl(U) \boxtimes R^{n-1} ) \cap (\partial_{i,0}^{*} -\partial_{i,\infty}^*) \cl \left( Z \right) \\
     &= \pr_* \sum_{i=1}^n (-1)^{i+1} ( \cl(U) \boxtimes R^{n-1} ) \cap \cl \left( (\partial_{i,0}^{*} -\partial_{i,\infty}^*) Z \right) \\
     &= \pr_* \left( ( \cl(U) \boxtimes R^{n-1}) \cap \cl(\partial Z) \right)
\end{align*}
%\begin{align*}
%  d\pr_* ( \cl(X) \times R^n \cap \cl(Z) ) 
%     &= \pr_* ( \cl(X) \times d R^n \cap \cl(Z) ) \\
%     &= \pr_* \sum_{i=1}^n (-1)^{i+1} \left( \cl(X) {\times} R^{i-1} {\times} d R^1 {\times} R^{n-i} \right)  \cap \cl(Z)  \\
%     &= \pr_* \sum_{i=1}^n (-1)^{i+1} \left( \cl(X) {\times} R^{i-1} {\times} \cl((z)) {\times} R^{n-i} \right)  \cap \cl(Z) 
%\intertext{%
%by the assumption. Note that $\cl((z))$ is the difference of the classes of the two points $0$ and $\infty$. By the projection formula one can identify the intersection $\cl((z_i)) \cap \cl(Z)$ with the pullback of $\cl(Z)$ along $\partial_{i,0} - \partial_{i,\infty}$, see also the appendix. Thus the above becomes }
%     &= \pr_* \sum_{i=1}^n (-1)^{i+1} ( \cl(X) {\times} R^{n-1} ) \cap \cl \left( (\partial_{i,0}^{*} -\partial_{i,\infty}^*) Z \right) \\
%     &= \pr_* \left( ( \cl(X) {\times} R^{n-1}) \cap \cl(\partial Z) \right)
%\end{align*}
by definition of the differential on Bloch's complex. 
 
We remark that by additivity of the intersection product, the first equality in the preceding calculation says that changing $R^n$ by a boundary changes the resulting regulator only by a boundary.

\paragraph*{Compatibility with products}
Next we show compatibility of the abstract regulator map with the product structures on higher Chow groups and on the complexes $C_\calD$. It suffices to consider integral coefficients; Then the result for arbitrary coefficients and alternating cycles follows from the linearity of the regulator. 
%
%The compatibility is shown twice on different levels of detail, starting with a vague version.
%
%Let $Z \in z^p(U,n)$ and $Z' \in z^q(U,m)$ be two higher Chow cycles that both lie in the domain of the regulator and meet transversally. Their intersection product is given by the pullback $\Delta^* ( Z \times Z')$ along the diagonal\footnote{More precisely, $\Delta = \text{switching map} \circ (\Delta_X \times \id_{\cube^{n+n'}})$.} of $U$.
%One gets
%\begin{align*}
%  r_\calD( Z \cdot Z') &= \pr_* ( \pr^* R^{n+m} \cap \Delta^* \cl(Z \times Z') ) \\
%                       &= \pr_* \Delta^* ( \pr^* R^{n+m} \cap \cl(Z \times Z') ) \\
%                       &= \pr_* \Delta^* ( (\pr^* R^n \cap \cl(Z)) \times (\pr^* R^{m} \cap \cl(Z')) ) \\
%                       &=  \Delta^* ( r_\calD(Z) \times r_\calD(Z') ).
%\end{align*} 

%Now we make the above proof more precise. 
Let $Z \in z^p(U,n)$ and $Z' \in z^q(U,m)$ be two higher Chow cycles that both lie in the domain of the regulator and meet transversally. 
Write $\Delta^{n,m}_X$ for the map $X \times \cube^{n+m} \to X \times \cube^n \times X \times \cube^m$ induced by the diagonal, and likewise $\Delta^{n,m}_U$. Then the intersection product of the two cycles is given by the pullback $Z \cdot Z' = (\Delta^{n,m}_U)^* ( Z \times Z')$. 

It suffices to show that $\Delta_* \big( r_\calD(Z) \cap r_\calD(Z') \big) = \Delta_* r_\calD(Z \cdot Z')$, since the diagonal is a closed embedding and its push-forward then injective. 
The right hand side is 
%Write $\Delta^{n,m}_X$ for the map $X \times \cube^{n+m} \to X \times \cube^n \times X \times \cube^m$ induced by the diagonal. 
%Then
%
%It suffices to show that $\Delta_* \big( r_\calD(Z) \cap r_\calD(Z') \big) = \Delta_* r_\calD(Z \cdot Z')$, since the diagonal is a closed embedding and its push-forward then injective. 
%Write $\Delta^{n,m}_X$ for the map $X \times \cube^{n+m} \to X \times \cube^n \times X \times \cube^m$ induced by the diagonal. 
%Then
\begin{align*}
  (\Delta_X)_* r_\calD( Z \cdot Z') &= (\Delta_X)_* (\pr_X)_* \Big( (\cl(U) \boxtimes R^{n+m}) \cap \cl(Z \cdot Z') \Big) \\ 
       &= (\pr_{X \times X})_* (\Delta^{n,m}_X)_* \Big( (\cl(U) \boxtimes R^{n+m}) \cap \cl(Z \cdot Z') \Big) \\
       &= (\pr_{X \times X})_* (\Delta^{n,m}_X)_* \Big( (\cl(U) \boxtimes R^{n+m}) \cap (\Delta_X^{n,m})^* \left[ \cl(Z) \times \cl(Z') \right] \Big) \\       
       &= (\pr_{X \times X})_* \Big( (\Delta^{n,m}_X)_* (\cl(U) \boxtimes R^{n+m}) \cap \cl(Z \times Z') \Big).       
\end{align*}
On the other side, for $\Delta_U$ the diagonal in $U \times U$, using reduction to the diagonal, the compatibility of $\boxtimes$ with push forward, $\cap$ and $\cl$, and the projection formula,  
\begin{align*}
   (\Delta_X)_* \left[ r_\calD(Z) \cap r_\calD(Z') \right] 
      &= \Big[ r_\calD(Z) \boxtimes r_\calD(Z') \Big] \cap \cl(\Delta_U) \\
      &= (\pr_{X \times X})_* \Big[ \big( (\cl(U) \boxtimes R^n ) \cap \cl(Z) \big) \boxtimes \big( (\cl(U) \boxtimes R^m ) \cap \cl(Z') \big) \Big] \cap \cl(\Delta_U) \\
      &= (\pr_{X \times X})_* \Big[ (\cl(U) \boxtimes R^n \times \cl(U) \times R^m) \cap (\cl(Z) \times \cl(Z')) \Big] \cap \cl(\Delta_U) \\      
      &= (\pr_{X \times X})_* \Big[ (\cl(U) \boxtimes R^n \times \cl(U) \times R^m) \cap \cl(Z \times Z') \cap (\pr_{X \times X})^* \cl(\Delta_U) \Big].     
\end{align*}
That this agrees with the expression above is easily verified as follows. Denote by $\tau$ the map that exchanges the two middle factors. Then use that intersection is compatible with biholomorphisms, the compatibility of $\cap, \boxtimes$ and that the cycle class $\cl(U \times U)$ is the identity for the intersection product to obtain
%\[
%  \left( \cl(U) \times R^n \times \cl(U) \times R^m \right) \cap (\pr_{X \times X})^*\cl(\Delta_U) = (\Delta_X^{n,m})_* (\cl(U) \times R^{n+m}).
%\]
%$\tau_{23}$ = switching map. 
\begin{align*}
  &\Big( \cl(U) \boxtimes R^n \boxtimes \cl(U) \boxtimes R^m \Big) \cap (\pr_{X \times X})^*\cl(\Delta_U) \\
  = &\Big( \cl(U) \boxtimes R^n \boxtimes \cl(U) \boxtimes R^m \Big) \cap \tau_* \Big( \cl(\Delta_U) \boxtimes \cl(\cube^{n+m}) \Big) \\
  = & \tau_* \Big[ \Big( \cl(U) \boxtimes \cl(U) \boxtimes R^n \boxtimes R^m \Big) \cap \Big( \cl(\Delta_U) \boxtimes \cl(\cube^{n+m}) \Big) \Big] \\  
  = & \tau_* \Big[ \cl(\Delta_U) \boxtimes R^n \boxtimes R^m \Big] \\    
  = & (\Delta_X^{n,m})_* (\cl(U) \boxtimes R^{n+m}).
\end{align*}
%The result follows\footnote{The first equality is proved in the appendix. A strict proof of the second statement is missing (but should be easy).} from $\left(\pr_{X \times X}^{X \times \cubeBar^n \times X \times \cubeBar^m} \right)^*  \cl(\Delta_U) = (\Delta^{n,m}_X)_* (\cl(U) \times \cubeBar^{n+m})$ and that the intersection of the latter chain with $X\times R^n \times X \times R^m$ is exactly $(\Delta^{n,m}_X)_*( \cl(U) \times R^{n+m})$.

Another way to phrase the above compatibilities of the regulator map is that the regulator induces a morphism between (suitable re-graded), partially defined differential graded algebras. Indeed, setting $\calN^r(U,p) := z^p(U, 2p-r)$, the regulator becomes a (partially defined) degree preserving map of complexes
\[
    r_\calD : \bigoplus \limits_p \calN^\bullet (U,p) \to \bigoplus \limits_p C_\calD^\bullet(X,D,\Z(p)).
\]
Compatibility with products implies that this is a map of dg algebras.

\paragraph{Functoriality}

Assume there is given a smooth morphism $f : U \to U'$ of algebraic manifolds induced by a smooth morphism of pairs $\overline f : (X,D) \to (X',D')$ of relative dimension $\delta$. 
The map $\overline f \times \id : X \times \cubeBar^n \to X' \times \cubeBar^n$ is proper, hence induces a push forward on the Deligne complexes (by functoriality).
If moreover $f$ is proper, then $f \times \id$ induces a push forward on higher Chow chains.
% and similarly $\overline f \times \id$ induces a push forward on the Deligne complexes (by functoriality).
%Then by functoriality $f \times \id_{\cube^n}$ and $\overline f \times \id_{\cubeBar^n}$ induce push forward morphisms on higher Chow chains and on the Deligne complexes respectively. 
\begin{lemma} For $f, \overline f$ proper, the regulator map is compatible with these push forwards, i.e., 
\[ 
\vcenter{
\xymatrix@C=4em{
     z^p(U,n) \ar[r]^{f_*} \ar[d]^{r_\calD}  &  z^{p-\delta}(U',n) \ar[d]^{r_\calD}  \\
     C_\calD^{2p-n}(X,D,\Z(p)) \ar[r]^{\overline{f}_*}  &  C_\calD^{2p-2\delta-n}(X',D',\Z(p-\delta))
}}
\text{ commutes.}
\]
\end{lemma} 

\begin{proof}
Let $Z \in c^p(U,n)$ such that $r_\calD(Z)$ is defined. Then
\begin{align*}
   \overline{f}_* r_\calD(Z) 
   &= \overline{f}_* \circ (\pr_X)_* \Big( (\cl(U) \boxtimes R^n) \cap \cl(Z) \Big) \\
   &= (\pr_{X'})_* \circ (\overline f \times \id_{\cubeBar^{n}})_* \Big( (\cl(U) \boxtimes R^n) \cap \cl(Z) \Big). 
\intertext{ Noting that $\cl(U) \boxtimes R^n = (\overline f \times \id_{\cubeBar^n})^* (\cl(U') \boxtimes R^n)$, the projection formula yields}
   &= (\pr_{X'})_* \Big( (\cl(U') \boxtimes R^n) \cap (\overline f \times \id_{\cubeBar^n})_* \cl(Z) \Big) \\
   &= (\pr_{X'})_* \Big( (\cl(U') \boxtimes R^n) \cap \cl( f_* Z ) \Big) \\   
   &= r_\calD(\cl( f_* Z )).
\end{align*}
The last equality holds because $\cl$ is compatible with push forward. 
\end{proof}

%\begin{remark}
%Any smooth algebraic morphism is a submersion (Hartshorne, \cite[III, Prop 10.4]{Hartshorne_AlgebraicGeometry}). In particular, $\overline{f}^*$ is defined. 
%\end{remark}

\paragraph{Compatibility with $S_{n}$ action}
The construction so far gives a map defined on higher Chow groups with $\Z$-coefficients and, by linearity, for coefficients in arbitrary rings.

In case that the exterior product $\boxtimes$ on the complexes $C_\calD$ is graded-commutative, we will see that $r_\calD(\alt(Z)) = r_\calD(Z)$. This simplifies the calculation of the regulator for those cycles that are given as the alternation of some cycle. In particular, it says that the image of an alternating cycle has integral coefficients. 

\begin{lemma} If the exterior product on $C_\calD$ is graded-commutative, then the diagram below commutes
\[
  \xymatrix@C=5em{
     \calN^\bullet(U,p)_A \ar[d]_{\alt} \ar[r]^{r_\calD}  &  C_\calD^\bullet(X,D,A(p)) \ar@{^(->}[d] \\
     \calN^\bullet(U,p)_{A_\Q}^{\alt} \ar[r]^{r_\calD}  &  C_\calD^\bullet(X,D,A_\Q(p)) 
  }
\]
%%%%%%%%%%%% with A = \Z : 
%\[
%  \xymatrix@C=5em{
%     \calN^\bullet(U,p) \ar[d]_\alt \ar[r]^{r_\calD}  &  C_\calD^\bullet(X,D,\Z(p)) \ar@{^(->}[d] \\
%     \calN^\bullet(U,p)_\Q^\alt \ar[r]^{r_\calD} \ar@{-->}[ru]  &  C_\calD^\bullet(X,D,\Q(p)) 
%  }
%\]
\end{lemma} 

\begin{proof}
The action of the symmetric group $S_n$ on $X \times \cubeBar^n$ induces an action on the complexes $C_\calD(X \times \cubeBar^n, D \boxtimes \mathbbm{1}^n, A)$ through functorial push forward. 
Because $\boxtimes$ is a graded-commutative exterior product, $R^n$ is invariant under the action of $S_n$. 

As a consequence (note that $g_*$ acts as the identity on $\cl(U) \boxtimes R^n$), 
\begin{align*}
   r_\calD( g Z ) 
      &= \pr_* \Big( (\cl(U) \boxtimes R^n) \cap \cl( g Z ) \Big) \\
      &= \pr_* \Big( (\cl(U) \boxtimes R^n) \cap g_* \cl( Z ) \Big) \\
      &= \pr_* \Big( g_* ( \cl(U) \boxtimes R^n \cap \cl( Z ) ) \Big)  \\
      &= \pr_* \Big( \cl(U) \boxtimes R^n \cap \cl( Z ) \Big)      
      = r_\calD(Z),
\end{align*}
i.e., the regulator is $S_n$-invariant.
\end{proof}

Another consequence of the above lemma is that the regulator is also compatible with the product on the alternating complexes. Thus the restriction of the regulator map to alternating cycles is a partially defined map of (partially defined) graded commutative algebras 
\[
    \bigoplus_p \calN^\bullet(U,p)^{\alt}_\Q \longrightarrow \bigoplus_p C_\calD^\bullet(X,D,\Q(p))
\]
if the product on $C_\calD$ is graded commutative.

\begin{remarks}
\item 
    By the moving lemma for higher Chow groups, the restriction induces a quasi-isomorphism of complexes $z^p(X,D,\bullet) := \frac{z^p(X,\bullet)}{i_* z^{p-1}(D,\bullet)} \to z^p(U,\bullet)$, and similar with subscript $\R$ (see \cite[5.9]{KLM_AbelJacobi}). Thus our construction can be given completely in the ''framework of pairs''.

\item Notice that the definition of the regulator map can be rewritten as 
      \[
           r_C ( Z ) = (\pr_X)_* \Big( \left(\pr_{\cubeBar^n}^{X \times \cubeBar^n}\right)^* (R^n) \cap \cl(Z) \Big).
      \]
      In this form, the regulator value of $Z$ can be seen as the image of the current $R^n$ under the (analytic) correspondence given by $Z$.
\item 
     The regulator is determined by $R^1$. 
     Conversely, $R^1$ can be recovered from the regulator on $U = \cube$ as the value of the diagonal $\Delta$, 
     \[
         R^1 = r_\calD(\Delta).
     \]
     In particular, the regulator is determined by giving a map $z^1(\cube,1) \to C^1_\calD(\cubeBar,1,\Z(1))$.
%\item      
%     One can argue why the abstract regulator $z^p(X,n) \to C_\calD^{2p-n}(X,\Z(p))$ should be of the form as given in section \ref{sec:abstract_regulator} as follows. 
%     Assuming that a cycle $Z$ is completely described by its associated Deligne cycle, then the regulator should factorize as a composition
%\[
%    z^p(X,n) \xrightarrow{\cl} C_\calD^{2p}(X \times \cube^n, \Z(p)) \xrightarrow{R^n} C_\calD^{2p-n}(X, \Z(p)).
%\]
%The maps $R^n$ can be seen as regulator maps on the Deligne complexes, defined for "higher Deligne cycles".
%It seems natural to require that these analytic regulators are associative in the sense that $R^{n+m} = R^n \circ R^m$. In other words, the $R^n$ are completely determined by the map
%\[
%    R^1 : C_\calD^{2p}(X \times \cube, \Z(p)) \to C_\calD^{2p-1}(X,\Z(p)).
%\] 
%If we further assume that this map has the form $c \mapsto (\pr_X)_* (\pr^* R^1 \cap c)$, that is, the map $R^1$ is induced by an element (also called) $R^1 \in C_\calD^1(\cube,\Z(1))$, then by induction the form of $R^n$ turns out to be 
%\begin{align*}
%   R^n(-) = (\pr_X)_* \Big( (\pr_{\cube})^*  (R^1)^{\times n} \cap -  \Big).
%\end{align*}
%This is exactly the form which has been assumed for the abstract regulator. 
\end{remarks}

\section{Regulator into the 3-term complex}
We now apply the above construction to the 3-term complex of currents from section \ref{sec:deligne_cohomology}. For $U$ a quasi-projective manifold with compactification $X$ and normal crossing boundary divisor $D$, it is the complex
%\[
%     \Tot \left( \calI(X,\Z(p)) \oplus F^p \sD(X) \xrightarrow{ -\delta + \iota} \sD(X) \right).
%\]
\[
     \Tot \left( \calI(X, D,\Z(p)) \oplus F^p \sD(X, \log D) \xrightarrow{ -\delta + \iota} \sD(X, \log D) \right).
\]
These complexes inherit the required functoriality properties from the functoriality of the underlying currents (componentwise applied). To have the correct twists in the coefficients after pushing forward along a morphism $f : X \to Y$, one has to multiply the componentwise push forward with an additional factor that involves the relative dimension $\rho = \dim_\C Y - \dim_\C X$. More precisely, one sets $f_*(a,b,c) = (2\pi i)^{\rho} (f_* a , f_* b , f_* c)$.

For an algebraic cycle $Z$ of codimension $p$ its associated Deligne cycle in the 3 term complex is given by the tuple $\cl(Z) = (2 \pi i)^p ([Z],[Z],0)$, with $[Z]$ being the integral $(p,p)$-current of integration over the non-singular part $Z_\reg$. More precisely, it is the simple extension of $[Z_\reg]$ to the whole $X \times \cubeBar^n$, as studied in \cite{Lelong_IntegrationSurUnEnsembleAnalytiqueComplexe}, \cite{Herrera_IntegrationOnASemianalyticSet}. The currents $[Z]$ are $d$-closed and hence the tuple is a cycle and represents a cohomology class.

On such a 3-term\footnote{Beilinson originally defined his product more general for $n$-term-complexes in his notes on 	absolute Hodge cohomology in 1983, which appeared 1 year earlier than the Russian original of \cite{BeilinsonHigherRegulatorsAndValuesOfLFunctions} complexes where he introduced Deligne cohomology (on the 3 term complex).} complex as above, Beilinson \cite[1.11]{NotesOnAbsoluteHodgeCohomology} found a whole family of products indexed by a real parameter $\alpha$, see also Esnault/Viehweg \citep{EsnaultViehweg}. 
%These complexes are equipped with a family of products, indexed by a real parameter $\alpha$. 
For a tuple $(a,b,c)$ of total degree $r$, its product with another tuple is expressed by the table
\[
\begin{array}{c|ccc}
  & \tilda & \tildb & \tildc \\ 
\hline 
a & a \cap \tilda & 0 & (-1)^r (1-\alpha) \cdot a \cap \tildc  \\ 
b & 0 & b \cap \tildb & (-1)^r \alpha \cdot b \cap \tildc \\ 
c & \alpha \cdot c \cap \tilda & (1-\alpha) \cdot c \cap \tildb & 0 \\ 
\end{array} 
\]
As the intersection product for currents is only partially defined, the same also holds for the above defined product on the 3-term complex. 
For the parameters $\alpha = 0$ and $\alpha = 1$ the product is associative, for $\alpha = \frac{1}{2}$ graded commutative.
Any two of these products are homotopic, so they all induce the same product on cohomology. \\
Replacing the intersection in the above table with the exterior product yields an exterior product on the 3-term complex. This allows reduction to the diagonal.

In order to apply the construction from the preceding section, we need to define a "base" element in $C_\calD^1(\cubeBar,1,\Z(1))$. For that, we copy from \cite[5.3]{KLM_AbelJacobi}. Choose the logarithm on $\P^1$ branched over $\Rm = [-\infty, 0]$, with $\Rm$ oriented in such a way that its boundary $\partial \Rm = 0 - \infty = (z)$ is the divisor of the coordinate function. Then define 
%
%To define\footnote{Here we use the same notation as in KLM. Maybe it is better to use the one as in section \ref{sec:abstract_regulator} ?? Then the lift of the regulator in section \ref{sec:regulator_to_holim} should be written as $\tilde{r_\calD}$.} the underlying "base" element for the regulator in $C_\calD^1(\cube,1)$ as
\[
    R^1 := \left( 2 \pi i [\Rm], [\dlog z],  [\log z] \right).
\]
From the formula of currents $d[\log z] = [\dlog z] - 2 \pi i [\Rm]$, as proved for example in \cite{KLM_AbelJacobi}, one obtains that the above element has differential $d R^1 = 2 \pi i \cdot ((z),(z),0)$, that is, the cycle associated to the divisor of the coordinate $z$.

%Here $AJ^1$ is chosen such that $d AJ^1 = 2 \pi i \cdot ((z),(z),0)$ is the class of the divisor of the coordinate $z$ (for this recall the formula of currents $d[\log z] = [\dlog z] - 2 \pi i [\R^-]$, as proved for example in \cite{KLM}).
%\[
%    AJ^1 := ( \R^-, \dlog(z)/(2\pi i), \log(z)/(2\pi i) )
%\]

%Next choose an associative product on this complex and define the $n$-fold exterior product in Deligne cohomology
%\[
%       AJ^n := AJ^1 \times \cdots \times AJ^1.
%\]
%This gives an element in the weight $n$ deligne complex on $\cube^n$.
%Now define the regulator as the map $\CH^p(X,n) \to H_\calD^{2p-n}(X,\Z(p))$ as follows. 
%Let $Z$ be an algebraic cycle in $X \times \cube^\bullet$. Interpret $Z$ as a (closed) element in the 3 term complex on $X \times \cube^\bullet$ via its associated Deligne cycle $(Z, \delta_Z, 0)$. 
%%Pull $AJ^n$ back along the projection map to obtain a current on $X \times \cube^n$. 
%Assume that $Z$ is in good position with respect to $AJ^n$ and define the Abel-Jacobi image of $Z$ to be 
%\[
%     AJ(Z) := (2\pi i)^{p-m} \cdot (\pi_X)_* \Big( (\pi_{\cube})^* AJ^n \cdot \{Z\}  \Big).
%\]
%This is an element in\footnote{$(\pi_X)_* \left( C_\calD^n(X\times \cube^n,\Z(n)) \cdot C^{2p}_\calD (X\times \cube^n,\Z(p)) \right) \subset (\pi_X)_* C_\calD^{n+2p} ( X \times \cube^n, \Z(n+p)) \subset C_\calD^{2p-n}(X, \Z(p) )$}  $C_\calD^{2p-n}(X,\Z(p))$.

\paragraph{The KLM-regulator}
Note that for the construction from section \ref{sec:abstract_regulator} to be applicable (well-defined), it suffices that the exterior product in question is associative. In particular, it applies to Beilinson's classical product for the parameter $\alpha = 0$.
%Since Beilinson's classical product for the parameter $\alpha = 0$ is associative, the construction from section \ref{sec:abstract_regulator} can be applied.  
In this case the product is given by the formula 
\[
     (a,b,c) \, \cap_0 \, (\tilda,\tildb,\tildc) 
     = \left( a \cap \tilda \,,\, b \cap \tildb \,,\, c \cap \tildb + (-1)^r a \cap \tildc \right)
\]
whenever the right hand side exists. 
Similarly for the exteior product. 
The $n$-th exterior power of $R^1$ is an element in the Deligne-Beilinson complex $C^n(\cubeBar^n,\mathbbm{1}^n,\Z(n))$. 
%If its components are denoted by $R^n = (T^n, \Omega^n, L^n)$, they are computed to be (with signs coming from the comparison of the alternating ($\wedge$) and the symmetric ($\times$) exterior product)
%\begin{align*}
%    T^n      &= (2 \pi i)^n (\Rm)^{\times n}  \\
%    \Omega^n &= \Omega(z_1, \ldots, z_n ) = (-1)^{\binom{n}{2}} [ \dlog z_1 \wedge \ldots \wedge \dlog z_n ] \\
%    L^n      &= L(z_1,\ldots, z_n) 
%               =  [\log z_1] \times \Omega( z_2, \ldots z_n) 
%                 - 2 \pi i [\Rm] \times L(z_2,\ldots, z_n) \\    
%              &=  \sum_{k=0}^{n-1} (- 2 \pi i)^k [ \underset{k \text{ times}}{\underbrace{\Rm \times \ldots \times \Rm}} ] \times [\log z_{k+1} ] \times [ \dlog z_{k+2} ] \times \ldots \times [ \dlog z_n ] \\
%              &=  \sum_{k=0}^{n-1} (-1)^{\binom{n-k-1}{2}} (- 2 \pi i)^k [ \Rm \times \ldots \times \Rm ] \times [\log z_{k+1} \dlog z_{k+2} \wedge \ldots \wedge \dlog z_n ] \\    
%              &=  \sum_{k=0}^{n-1} (-1)^{k(n-k) + \binom{n-k-1}{2}} (2 \pi i)^k [ \Rm \times \ldots \times \Rm ] \wedge [\log z_{k+1} \dlog z_{k+2} \wedge \ldots \wedge \dlog z_n ] \\                              
%              &=  (-1)^{\binom{n-1}{2}} \sum_{k=0}^{n-1} (-1)^{\binom{k+1}{2}} (2 \pi i)^k [ \Rm \times \ldots \times \Rm ] \wedge [\log z_{k+1} \dlog z_{k+2} \wedge \ldots \wedge \dlog z_n ].
%\end{align*}
Its components $R^n = (T^n, \Omega^n, L^n)$ can be computed to be (with signs coming from the comparison of the "homological" ($\times$) and the "cohomological" ($\boxtimes$) exterior product)
\begin{align*}
    T^n      &= (-1)^{\binom{n}{2}} (2 \pi i)^n [(\Rm)^{\times n}]  \\
    \Omega^n &= \Omega(z_1, \ldots, z_n ) = [ \dlog z_1 \wedge \ldots \wedge \dlog z_n ] \\
    L^n      &= L(z_1,\ldots, z_n) 
               =  [\log z_1] \boxtimes \Omega( z_2, \ldots z_n) 
                 - 2 \pi i [\Rm] \boxtimes L(z_2,\ldots, z_n) \\    
              &=  \sum_{k=0}^{n-1} (- 2 \pi i)^k [\Rm]^{\boxtimes k} \boxtimes [\log z_{k+1} ] \boxtimes [ \dlog z_{k+2} \wedge \ldots \wedge \dlog z_n ] \\
              &=  \sum_{k=0}^{n-1} (-1)^{\binom{k}{2}} (- 2 \pi i)^k [ \underset{k \text{ times}}{\underbrace{\Rm \times \ldots \times \Rm}} ] \boxtimes [\log z_{k+1} ] \boxtimes [ \dlog z_{k+2} \wedge \ldots \wedge \dlog z_n ] \\
              &=  \sum_{k=0}^{n-1} (-1)^{\binom{k+1}{2}} (2 \pi i)^k [ \underset{\text{first } k \text{ coordinates}}{\underbrace{\Rm \times \ldots \times \Rm}} ] \wedge (\log z_{k+1} \dlog z_{k+2} \wedge \ldots \wedge \dlog z_n ).                  
\end{align*}
For example, the first exterior powers of $R^1$ are 
\begin{align*}
    R^2 &= \Big( -(2 \pi i)^2 [\Rm \times \Rm], [\dlog z_1 \wedge \dlog z_2], [\log z_1 \dlog z_2] - 2 \pi i [\Rm] \boxtimes [\log z_2] \Big) \\
    R^3 &= \Big( -(2 \pi i)^3 [\Rm \times \Rm \times \Rm], [\dlog z_1 \wedge \dlog z_2 \wedge \dlog z_3], \\
    & \qquad \ [\log z_1 \dlog z_2 \wedge \dlog z_3] - 2 \pi i [\Rm] \boxtimes [\log z_2 \dlog z_3] - (2 \pi i)^2 [\Rm \times \Rm] \boxtimes [\log z_3] \Big) .
\end{align*}
Let $Z \in z^p(U,n)$ be a higher Chow cycle. 
Since $\bigcap_0$-multiplication with $([Z],[Z],0)$ is just componentwise intersection with the current represented by $Z_\reg$, the regulator of $Z$ is (again up to signs) 
\[
     r_C(Z) = (2\pi i)^{p-n} \cdot (T_Z,\Omega_Z, L_Z) 
\]
where
\begin{itemize}
\item $T_Z = (2 \pi i)^n (\pr_X)_* \Big(([X] \boxtimes T^n) \cap [Z] \Big)$, 
\item $\Omega_Z = (\pr_X)_* \Big( ([X] \boxtimes \Omega^n) \cap [Z] \Big)$, 
\item $L_Z = (\pr_X)_* \Big(  ([X] \boxtimes L^n) \cap [Z] \Big)$.
\end{itemize}
%\begin{itemize}
%\item $T_Z = (2 \pi i)^n (\pr_X)_* \Big( \pr_{\cubeBar^n})^* T^n \cap [Z] \Big)$, 
%\item $\Omega_Z = (\pr_X)_* \Big( \pr_{\cubeBar^n})^* \Omega^n \cap [Z] \Big)$ 
%\item $L_Z = (\pr_X)_* \Big(  (\pr_{\cubeBar^n})^* L^n \cap [Z] \Big)$.
%\end{itemize}
%where $T_Z = (2 \pi i)^n (\pr_X)_* \Big(([X] \times T^n) \cap [Z] \Big)$, $\Omega_Z = (\pi^Z_X)_* (\pi^Z_{\cube^n})^* \Omega^n$ and $L_Z = (\pi^Z_X)_* (\pi^Z_{\cube^n})^* L^n$.
Note that for these currents to be well-defined one needs that $Z$ intersects the $\Rm$ components properly, i.e., one has to restrict to cycles in $z_\R^p(U,n)$ to get a everywhere defined map.

This map agrees (up to $2\pi i$-factors and various signs) with the map given in \cite{KLM_AbelJacobi}.

\begin{remark}
\begin{itemize}
\item More correctly one should replace the $Z$ in the above formula by $Z_\reg$.
      We make the convention that before actually performing the integration, 
      1) the domain of undefinedness of the integrand should be removed from the integration domain 
      2) the integration domain should be replaced by its manifold points.

      %Under the ''proper intersection'' hypthesis, the removed part has higher codimension and thus the integral doesn't change.       
      
\item Since the product $\bigcap_0$ is graded-commutative only up to homotopy, this regulator map is no strictly commutative (i.e. on level of complexes) map of graded-commutative dga's.
\item One could similarly proceed with an arbitrary $\bigcap_\alpha$, but since in this case the product need not be associative, one has to choose explicitly how to evaluate the iterated products. Different choices give rise to different (though homological equivalent) regulators. 
\end{itemize}
\end{remark}

\section{The regulator into $P_\calD$}  \label{sec:regulator_to_holim}
In order to get a strict graded-commutative regulator map, we will now apply the construction to the complexes $P_\calD(X,D,A(p))$.
Recall that for $(X, D)$ a good compactification of $U$ one has 
%\[
%    P_\calD(X,p) = \left\{ w \otimes T \in \Omega(x) \otimes \calD(\Xbar) \mid w(0) T \in \calI(\Xbar,D,\Z(p)), w(1) T \in F^p \calD(\Xbar, \log D) \right\}.
%\]
\[
    P_\calD(X,D,A(p)) := \left\{ w \otimes T \in \Omega(x) \otimes \calD(X, \log D) \text{ such that } 
    \genfrac{}{}{0pt}{}{ w(0) T \in \calI(X, D, A(p)),} { w(1) T \in F^p \calD(X, \log D) } \right\}.
\]
%Recall that the weight $p$ Deligne cohomology for projective $X$ can be computed by the complex
%\[
%    P_\calD(X,p) = \left\{ w \otimes T \in \Omega(x) \otimes \calD(X) \mid w(0) T \in \calI(X,\Z(p)), w(1) T \in F^p \calD(X) \right\}.
%\]
These complexes inherit all the functorial properties from the complex of currents by letting a morphism act trivially on $\Omega(x)$. To get the correct coefficients, the push forward is twisted in exactly the same way as the push forward for $C_\calD$. 
%As in the 3-term complex, functoriality is inherited from the functoriality of $\calD$ and the identity on $\Omega(x)$, where the pushforward is again twisted to get the correct coefficients. 
A codimension $p$ cycle $Z$ in $X$ is represented by the constant path $\cl(Z) = (2 \pi i)^p [Z]$. 

They are equipped with a $\C$-linear exterior product coming from the wedge product on $\Omega(x)$ and the exterior product of currents. Explicitly, $(w \otimes T) \boxtimes (\eta \otimes S) := (-1)^{|S||\eta|} w \wedge \eta \otimes (S \boxtimes T)$.
Similarly, the intersection product is defined by replacing the exterior product in the above formula with the $\cap$ product. It is defined whenever the intersection of the underlying currents is defined and has the correct type (lies in $P_\calD$). Each pair of cohomology classes can be represented by cycles whose intersection in $P_\calD$ exists. 
The two products are both associative and graded-commutative in their sense.
With these definitions, reduction to the diagonal holds and compatibility of $\boxtimes$ and $\cap$ is satisfied for currents coming from geometry\footnote{See appendix. More precisely, one has $(S \cap T) \boxtimes (P \cap Q) = (-1)^{|T||P|} (S \boxtimes P) \cap (T \boxtimes Q)$ if all intersections exist. If $S, P$ or $T,Q$ are both coming from geometry, they have even degree and the sign vanishes.}
% $[(w_1 {\wedge} T_1) {\times} (w_2 {\wedge} T_2)] \cap [(\eta_1 {\wedge} S_1) {\times} (\eta_2 {\wedge} S_2)] = (-1)^{|w_2 T_2||\eta_1 S_1|+|T_1||S_2|+|T_2||S_1|}[(w_1 {\wedge} T_1) \cap (\eta_1 {\wedge} S_1) ] \times [(w_2 {\wedge} T_2) \cap (\eta_2 {\wedge} S_2)]$. If $w_i {\wedge} T_i$, $i = 1,2$ are coming from geometry, all $w_i$ and $T_i$ have even degree. Same if $\eta_i {\wedge} S_i$ are coming from geometry.}.

For the underlying element of the regulator we choose the element in $P_\calD^1(\P^1,1,\Z(1))$ defined as 
%the log current in $\P^1 \supset \cube^1$
%    R^1 := (1-x) \R^- + x \scriptfrac{\dlog z}{2\pi i} + dx \scriptfrac{\log z}{2 \pi i}
\[
    R^1 := (1-x) (2 \pi i) \Rm + x  \dlog z + dx \log z.
\]
%where $\R^- = \R_{\leq 0} \subset \P^1(\C)$ is oriented in such a way that $\partial \R^- = 0 - \infty = (z)$.
Its relation to the element that defines the regulator for the 3-term complex is described in \ref{subsec:comparison_of_regulators}.
Again using $d[\log z] = \dlog z - 2 \pi i [\Rm]$ one verifies that $dR^1 = 2 \pi i \cdot [0 - \infty] = \cl((z))$.

Exterior multiplication yields $R^n = R^1 \boxtimes \ldots \boxtimes  R^1$, which for small values of $n$ is (with $2\pi i$-factors before $\Rm$ omitted) given by the currents on $\cubeBar^n$:
%\begin{align*}
%    R^2 = &(1-x)^2 \Rm {\times} \Rm + x^2 \dlog {\times} \dlog 
%        + x (1-x) [\Rm {\times} \dlog + \dlog {\times} \Rm ]  \\
%        &+ x dx [\log {\times} \dlog - \dlog {\times} \log ] 
%        + (1-x) dx [ \log {\times} \Rm - \Rm {\times} \log ]  \\[\baselineskip]
%    R^3 = & (1-x)^3 \Rm {\times} \Rm {\times} \Rm 
%        + (1-x)^2 x [ \Rm {\times} \Rm {\times} \dlog + \Rm {\times} \dlog {\times} \Rm + \dlog {\times} \Rm {\times} \Rm] \\
%        &+ (1-x) x^2 [\Rm {\times} \dlog {\times} \dlog + \dlog {\times} \Rm {\times} \dlog + \dlog {\times} \dlog {\times} \Rm] 
%         + x^3 \dlog {\times} \dlog {\times} \dlog \\ 
%        &+ x^2 dx [\log {\times} \dlog {\times} \dlog - \dlog {\times} \log {\times} \dlog + \dlog {\times} \dlog {\times} \log] \\
%        &+ (1-x)x dx [\log {\times} \Rm {\times} \dlog - \Rm {\times} \log {\times} \dlog + \Rm {\times} \dlog {\times} \log ] \\
%        &+ (1-x)x dx [\log {\times} \dlog {\times} \Rm - \dlog {\times} \log {\times} \Rm + \dlog {\times} \Rm {\times} \log] \\ 
%        &+ (1-x)^2 dx [\log {\times} \Rm {\times} \Rm - \Rm {\times} \log {\times} \Rm + \Rm {\times} \Rm {\times} \log] 
%\end{align*}
\begin{align*}
    R^2 = &(1-x)^2 \Rm \boxtimes \Rm + x^2 \dlog \boxtimes \dlog 
        + x (1-x) [\Rm \boxtimes \dlog + \dlog \boxtimes \Rm ]  \\
        &+ x dx [\log \boxtimes \dlog - \dlog \boxtimes \log ] 
        + (1-x) dx [ \log \boxtimes \Rm - \Rm \boxtimes \log ]  \\[\baselineskip]
    R^3 = & (1-x)^3 \Rm \boxtimes \Rm \boxtimes \Rm 
        + (1-x)^2 x [ \Rm \boxtimes \Rm \boxtimes \dlog + \Rm \boxtimes \dlog \boxtimes \Rm + \dlog \boxtimes \Rm \boxtimes \Rm] \\
        &+ (1-x) x^2 [\Rm \boxtimes \dlog \boxtimes \dlog + \dlog \boxtimes \Rm \boxtimes \dlog + \dlog \boxtimes \dlog \boxtimes \Rm] 
         + x^3 \dlog \boxtimes \dlog \boxtimes \dlog \\ 
        &+ x^2 dx [\log \boxtimes \dlog \boxtimes \dlog - \dlog \boxtimes \log \boxtimes \dlog + \dlog \boxtimes \dlog \boxtimes \log] \\
        &+ (1-x)x dx [\log \boxtimes \Rm \boxtimes \dlog - \Rm \boxtimes \log \boxtimes \dlog + \Rm \boxtimes \dlog \boxtimes \log ] \\
        &+ (1-x)x dx [\log \boxtimes \dlog \boxtimes \Rm - \dlog \boxtimes \log \boxtimes \Rm + \dlog \boxtimes \Rm \boxtimes \log] \\ 
        &+ (1-x)^2 dx [\log \boxtimes \Rm \boxtimes \Rm - \Rm \boxtimes \log \boxtimes \Rm + \Rm \boxtimes \Rm \boxtimes \log] 
\end{align*}
In general $R^n = R^n_0 + R^n_1$ where $R^n_i$ consists of those summands of $R^n$ whose $dx$-degree is $i$. They satisfy 
%\begin{align*}
%   R^{n+1}_0 &= (1-x) 2 \pi i R^n_0 {\boxtimes} \Rm + x R^n_0 {\boxtimes} \dlog \\
%   R^{n+1}_1 &= (-1)^n dx R^n_0 {\boxtimes} \log + (1-x) 2 \pi i R^n_1 {\boxtimes} \Rm + x R^n_1 {\boxtimes} \dlog 	
%\end{align*}
\begin{align*}
   R^{n+1}_0 &= (1-x) 2 \pi i R^n_0 \boxtimes \Rm + x R^n_0 \boxtimes \dlog \\
   R^{n+1}_1 &= (-1)^n dx R^n_0 \boxtimes \log + (1-x) 2 \pi i R^n_1 \boxtimes \Rm + x R^n_1 \boxtimes \dlog 	
\end{align*}
Thus $R^n_0$ consists of $2^n$ summands: all possible combinations built from $\dlog$ and $\Rm$. The degree one part $R^n_1$ grows faster: It consists of $n 2^{n-1}$ summands, $2^{n-1}$ for each position where the $\log$ term can be placed.

The resulting regulator map $r_P :  z^p_\R(U,n) \to P_\calD^{2p-n}(X,D,\Z(p))$ is then given by
\begin{align*}
      r_P(Z) = (2\pi i)^{p} \pr_* \left( \pr^* R^n \cap [Z] \right).
\end{align*}

%\begin{align*}
%      CH^p(X,n) \to P_\calD^{2p-n}(X,p) \\
%      Z \mapsto (2\pi i)^{p-n} \pr_* ( \pr^* R^n \cdot [Z] ).
%\end{align*}

\subsection{Comparison} \label{subsec:comparison_of_regulators}
We now compare the two versions of the regulator map. The target complexes of these maps are related by the evaluation homomorphisms
\begin{equation}
  %\label{ev} \tag{ev}
  \begin{aligned} 
  ev : P_\calD^n(X,D,A(p)) &\to C_\calD^n(X,D,A(p)) \\
  w \otimes T  &\mapsto \left( w(0) T, w(1) T, \int_{[0,1]} w \cdot T \right). 
  \end{aligned}
\end{equation} 
When $n$ varies, these maps give rise to a morphism of complexes, which turns out to be a quasi-isomorphism whenever $\Q \subset A$. This is content of the following lemma, that seems to be well-known to experts.  
\begin{lemma}
If $\Q \subset A$, the morphism $ev$ is a quasi-isomorphism. A quasi-inverse is induced by the maps 
\begin{align*}
  s : C^n_\calD(X,D,A(p)) &\to P_\calD^n(X,D,A(p)) \\
      (a,b,c)        &\mapsto (1-x) \otimes a + x \otimes b + dx \otimes c .
\end{align*}
\end{lemma}

\begin{proof}
$s$ is indeed compatible with differentials, thus gives a map of complexes. 
It is obvious that $s$ splits the map $ev$, that is $ev \circ s = \id$.
It suffices to show that the map $s \circ ev : P_\calD \to P_\calD$ is homotopic to the identity. Such a homotopy was given by Burgos Gil during a summer school in Freiburg 2013. Define the homotopy $h : P_\calD \to P_\calD [-1]$ by 
\[
    h( w \otimes c) = x \int_{[0,1]} w \otimes c - \int_{[0,x]} w \otimes c.
\]
This map is well defined because of $\Q \subset A$ (the occuring integrals in general have rational coefficients).
One checks that $dh+hd = s \circ ev - \id$. In particular, $s \circ ev = \id$ on cohomology and so $ev$ is a quasi-isomorphism.
\end{proof}
Note that the underlying elements of the two regulators ($R^1_{P}$ and $R^1_{C}$, say) can be obtained from each other by applying the map $ev$ resp. its splitting. 

Note that moreover $ev$ is just a map of vector spaces (not of algebras). Transporting the product in $P_\calD$ to $C_\calD$ using the splitting however gives the product $\bigcap_{1/2}$. Thus on homology classes the product on $P_\calD$ is equal to $\bigcap_{1/2}$ and hence to $\bigcap_0$, since they are homotopic. That means that $ev$ induces an isomorphism of algebras on homology. 
%
%
%\vspace{2cm} 
%For each $n$ these maps give rise to a morphism of complexes, which turns out to be a quasi-isomorphism whenever $\Q \subset A$. In fact, there is a split of $ev$ given by sending a triple $(a,b,c)$ to the element $(1-x) \, a + x \, b + dx \, c$. This is a quasi-inverse of $ev$ and the underlying elements of the two regulators ($R^1_{P}$ and $R^1_{C}$, say) can be obtained from each other by applying the map $ev$ resp. its splitting. 
%
%Note that $ev$ is just a map of vector spaces (not of algebras). Transporting the product in $P_\calD$ to $C_\calD$ using the splitting however gives the product $\bigcap_{1/2}$. Thus on homology classes the product on $P_\calD$ is equal to $\bigcap_{1/2}$ and hence $\bigcap_0$, since they are homotopic. That means that $ev$ induces an isomorphism of algebras on homology. 

Applying this to the construction of the two regulators, one gets that the two elements $R^n_{C}, R^n_{P}$ underlying the constructions differ only by a boundary. This means that $ev(R^n_{P})$ equals $R^n_{C}$ up to boundaries, and that $s(R^n_{P})$ equals $R^n_{P}$ up to boundaries.
Using this, we can show that the two regulator are isomorphic on cohomology: 
%Since furthermore the cycle maps and pullback/push forward are compatible with $ev$ and the splitting, the two regulators are equal on homology. 
%Put together, we have the 
\begin{lemma} \label{lemma:comparison_rP_and_rC_cohomologically}
For $\Q \subset A$, the non-commutative diagram of complexes 
\[
\xymatrix{
  z_\R^p(U,\bullet)_A \ar[r]^{r_{C}} \ar[d]_{r_{P}} &  C_\calD^{2p-\bullet}(X,D,A(p)) \ar@{<->}^{qIso}[dl] \\
  P^{2p-\bullet}_\calD(X,D,A(p))
}
\]
commutes after passage to (co-)homology.
\end{lemma} 

\begin{proof}
%$ev$ is compatible with push forward. Although $ev$ is in general not compatible with the exterior and the intersection product, one verifies easily that this is indeed true for the exterior product / intersection with cycles classes. 
%For example, $ev( W \cap \cl(Z) ) = ev(W) \cap_\alpha \cl(Z)$ for all $W \in P_\calD$, $\alpha \in [0,1]$ and all higher Chow cycles $Z$. Note that the first $\cl$ denotes the cycle map into $P_\calD$ and the second is the cycle map into $C_\calD$. 

$ev$ is compatible with push forward and, as one verifies easily, $ev$ is compatible with the intersection/exterior product with cycle classes. For example, $ev( W \cap \cl(Z) ) = ev(W) \cap_\alpha \cl(Z)$ for all $W \in P_\calD$, $\alpha \in \R$ and all higher Chow cycles $Z$. Note that the first $\cl$ denotes the cycle map into $P_\calD$ and the second is the cycle map into $C_\calD$. 
Using this,
\begin{align*}
   ev \circ \pr_* \Big( (\cl(U) \boxtimes R^n_P) \cap \cl(Z) \Big) 
         &= \pr_* \Big( ev( \cl(U) \boxtimes R^n_P) \cap_0 \cl(Z) \Big) \\
         &= \pr_* \Big( \cl(U) \boxtimes ev(R_P^n) \cap_0 \cl(Z) \\         
         &= \pr_* \Big( \cl(U) \boxtimes (R^n_C + \text{boundary}) \cap_0 \cl(Z) \Big) \\
         &= \pr_* \Big( (\cl(U) \boxtimes R^n_C) \cap_0 \cl(Z) \Big) + \text{boundary}.
\end{align*}
Thus one obtains $ev( r_P(Z) ) \equiv r_C(Z)$ on cohomology. 
\end{proof}

The above proof can be refined to give a result for cycles instead of merely cycle classes. 
The key observation is the following .
\begin{lemma} \label{lemma:comparison_RP_and_RC}
\[
     ev(R^n_P) = \alt_*(R^n_C).
\] 
\end{lemma}

\begin{proof} 
Compare the first two and the last component of the two triples separately. 
\begin{itemize}
\item Evaluating $x = 0$ in $R^n_P$ makes all summands to zero except $(2 \pi i)^n [\Rm]^{\boxtimes n}$, which is just the first component of $R_C^n$. 
      Since this is symmetric, it is also equal to the first component of $\alt_* R_C^n$.
      Similar for the second component. 

\item Every summand of $L^n$ has the form (we omit the $(2\pi i)^k$ factor)
      \[
           M =  (\Rm)^{\boxtimes k} \boxtimes \log \boxtimes \dlog^{\boxtimes(n-k-1)}. 
      \]      
      All these terms occur as well as a summand in $R^n_P$ with some 1-form coefficient (in the variable $x$). 
      Moreover, the $S_n$-orbits of the terms $M$ (for varying $k$) form a partition of $(R^n_P)_1$ (= terms with dx).
      That is, every summand "with dx" in $R_P^n$ has the form $\omega \otimes g_* M$ for some permutation $g \in S_n$.
      We have to compare the coefficients before $g_* M$ that occur $ev(R_P^n)$ and in $\alt_* R_C^n$.       

      First note that every element in the $S_n$-orbit of $M$ has the same coefficient-form $\omega$ whose integral is
      \[
            \int_0^1 (1-x)^k x^{n-k-1} dx = \frac{k! (n-k-1)!}{n!}
      \] 
      as follows from induction -- or using properties of the beta function. 
      
      On the other hand, the coefficient before $g_* M$ in $\alt M$ is equal to
      \[
          \frac{|\text{stabilizer of }M|}{|S_n|} = \frac{k! (n-k-1)!}{n!}.
      \]           
\end{itemize}
This proves the lemma.
\end{proof}

\begin{lemma} \label{lemma:comparison_regP_and_regC}
For $Z \in z^p_\R(U,n)$ and $\Q \subset A$, one has an equality in $C_\calD(X,D,A(p))$
\begin{align*}
      ev( r_P(Z) ) = r_C( \alt Z ).
\end{align*}
Moreover, for any $A$, the first two components of $ev(r_P(Z))$ and $r_C(Z)$ are equal, that is,
\[
 ev( r_P(Z) ) = ( r_C(Z)_0, r_C(Z)_1, \text{ rest } ).
\]
\end{lemma}
\begin{proof}
The second statement follows immediately, as indicated in the proof of \ref{lemma:comparison_RP_and_RC}, from the definitions and is omitted. 
Using the lemma \ref{lemma:comparison_RP_and_RC}, the first statement can be proven as follows. Starting similar to the proof of lemma \ref{lemma:comparison_rP_and_rC_cohomologically}, compute 
\begin{align*}
   ev \circ r_P(Z) 
   &= \pr_* \Big( ev( \cl(U) \boxtimes R_P^n) \cap_0 \cl(Z) \Big) \\
   &= \pr_* \Big( (\cl(U) \boxtimes ev(R_P^n)) \cap_0 \cl(Z) \Big) \\
   &= \pr_* \Big( (\cl(U) \boxtimes \alt_*(R_C^n)) \cap_0 \cl(Z) \Big) \\
   &= \pr_* \alt_* \Big( ( \cl(U) \boxtimes R_C^n) \cap_0 \alt_* (\cl(Z)) \Big). 
\end{align*}
Observing that $\pr \circ \, \alt = \pr$ and $\alt_* \cl(Z) = \cl(\alt Z)$ finishes the proof.
\end{proof}

\begin{remark} Lemma \ref{lemma:comparison_RP_and_RC} can be sharpened such that $ev(R_P^n) = (\id,\id,\alt) (R_C^n)$ even with integral coefficients. Similarly, Lemma \ref{lemma:comparison_regP_and_regC} has an version for integral coefficients that says $ev(r_P(Z)) = (2 \pi i)^p \pr_* ( (\id,\id,\alt)_* (\cl(U) \boxtimes R_C^n) \cap_0 \cl(Z) )$. 
\end{remark}

%When restricted to alternating cycles however, the above diagram already commutes on chain level.
%\begin{lemma} \label{lemma:comparison_regP_and_regC}
%For $Z \in z^p_\R(U,n)$ 
%\begin{align*}
%     r_C( \alt Z ) =  ev( r_P(Z) ) = ( r_C(Z), r_C(Z), \text{ rest } ).
%\end{align*}
%
%In particular, if $Z$ is $S_n$-invariant, 
%\[
%    ev( r_P(Z) ) = r_C(Z).
%\]
%\end{lemma}
%
%\begin{proof}
%The second equality follows easily from the definitions and is omitted. 
%The first equality will follow from 
%\begin{claim}
%\[
%     ev(R^n_P) = \alt_*(R^n_C).
%\]
%\end{claim}
%Using the claim, the lemma can be proven as follows. Starting similar to the proof above, compute 
%\begin{align*}
%   ev \circ r_P(Z) 
%   &= \pr_* \Big( ev( \cl(U) \boxtimes R_P^n) \cap_0 \cl(Z) \Big) \\
%   &= \pr_* \Big( (\cl(U) \boxtimes ev(R_P^n)) \cap_0 \cl(Z) \Big) \\
%   &= \pr_* \Big( (\cl(U) \boxtimes \alt_*(R_C^n)) \cap_0 \cl(Z) \Big) \\
%   &= \pr_* \alt_* \Big( ( \cl(U) \boxtimes R_C^n) \cap_0 \alt_* (\cl(Z)) \Big). 
%\end{align*}
%Observing that $\pr \circ \, \alt = \pr$ and $\alt_* \cl(Z) = \cl(\alt Z)$ finishes the proof.
%\end{proof}

\paragraph*{Aside: Interpretation of Beilinson's product}
The quasi-isomorphism between $P_\calD$ and $C_\calD$ leads to a geometric interpretation of Beilinson's products on the latter. 

Think of an element $(a,b,c)$ in $C_\calD$ as the startpoint $a$ and endpoint $b$ of a path with $c$ the line segment connecting $a$ and $b$ (oriented from $a \to b$). 
Given two such triples, one can form the cross product of the two paths, getting a square as drawn below
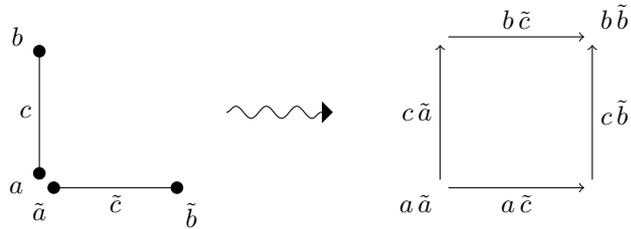
\begin{equation*} 
\begin{tikzpicture}[auto, inner sep=1mm]

   %first diagram
    \begin{scope}[xshift=-70]
    %Draw points
    %\node[circle,draw=blue!50,fill=blue!20] (a) at (-1,1) {$b \tilda$};
    \node (a) at (-1,1) {};
    \node (b) at (1,1) {};
    \node (c) at (1,-1) {};
    \node (d) at (-1,-1) {};
    
    %Add labels 
    \node[left=3pt] at (a) {$b$};
    \node[below=3pt] at (c) {$\tildb$};
    \node[below=3pt] at (d) {$\tilda$};
    \node[left=3pt] at (d) {$a$};    
    
    %Connect points
    \draw [*-*] (c) to node {$\tildc$} (d);
    \draw [*-*] (d) to node {$c$} (a);      
    \end{scope}

    % leadsto-arrow between diagramms
    \draw [-triangle 90,decorate,decoration={snake,amplitude=.8mm,segment length=4mm,post length=1mm}] (-1,0) -- (0.4,0);
   
   % second diagram
    \begin{scope}[xshift=80]
    %Draw points
    \node (a) at (-1,1) {};
    \node (b) at (1,1) {};
    \node (c) at (1,-1) {};
    \node (d) at (-1,-1) {};
    
    %Connect points
    \draw [->] (a) to node {$b  \,  \tildc$} (b);
    \draw [<-] (b) to node {$c  \,  \tildb$} (c);
    \draw [<-] (c) to node {$a  \,  \tildc$} (d);
    \draw [->] (d) to node {$c  \,  \tilda$} (a);  
    
    % Label endpoints
    \node[below left]  at (d) {$a \, \tilda$};
    \node[above right] at (b) {$b \, \tildb$};
         
    \end{scope}
\end{tikzpicture}
\end{equation*}

Now there are two possible ways to extract a path from the new startpoint $a \tilda$ to the new endpoint $b \tildb$ out of this diagram. 
Each is as good as the other and one has the freedom to combine them as one wishes using a parameter $\alpha$. 

For example, give the left upper path the weight $\alpha$ and the right lower path weight $1-\alpha$. Define the combined path to be
\[
    \alpha \cdot [ c \tilda \pm b \tildc ] + (1-\alpha) \cdot [ \pm a \tildc + c \tildb]   
\]
where we decorate each horizontal line segment with a sign $\pm = (-1)^r$. 
This in turn corresponds to a triple which is exactly the $\bigcap_\alpha$ product of $(a,b,c)$ with $(\tilda, \tildb, \tildc)$.

\begin{remark}
\begin{itemize}
\item There is no satisfactory geometric explanation for this sign $(-1)^r$. Maybe it should be assigned to the endpoints $a,b$ of the vertical path because $r = \deg a = \deg b$ is the degree of these points. In fact, the sign occurs when $a$ or $b$ are passing $\tildc$.
\item With this geometric construction in mind, the associativity and grad-commutativity properties of Beilinson's products are verified easily (e.g. by looking at a cube).
\end{itemize}
\end{remark}

\section{Examples}

\subsection{Regulator of a point}
Consider the special case of $U = X = \Spec \C$ being a point and coefficients $A = \Z$.
%For this, let $Z \subset \cube^n$ be a higher Chow cycle meeting the real boundaries properly. 
The space of currents over a point identifies with the field of complex numbers $\C$; a complex number $\lambda$ corresponds to multiplication with $\lambda$ (i.e. the distribution $\lambda [\pt]$).  
%The complex $P_\calD(\pt,\Z(p))$ then becomes the set of complex valued paths $w \in \C \otimes \Lambda [x,dx]$ with $w(0) \in \Z(p)$ and $w(1) = 0$ if $p > 0$. 
The complex $P_\calD(\pt,\Z(p))$ then becomes a subset of complex valued paths $\Omega_\C(x) = \C[x , dx]$, 
\[
      P_\calD(\pt,\Z(p)) = \left\{ w \in \C [x,dx] \ \middle|\  \parbox{0.2 \linewidth}{ $w(0) \in \Z(p)$ and \\ $w(1) = 0$ if $p > 0$} \right\}.
\] 

%\paragraph{The regulator}
The regulator into $P_\calD$ then is given by maps $z_\R^p(\pt,n) \to ( \Lambda[x,dx] \otimes \C )^ {2p-n}$. 
In particular, there are a priori only two non-trivial cases to consider: $n=2p$ and $n = 2p-1$. In each other case, the regulator value of a higher Chow cycle $Z$ is zero, because the intersection of $Z$ with $R^n$ will have a degree which is too high or low to survive the integration step underlying the push forward to the point. 
The two cases $n=2p$ and $n=2p-1$ correspond to the two summands of the decomposition $R^n = R^n_0 + R^n_1$: the regulator 
%\begin{align*}
%    r_\calD(Z) &= (2 \pi i)^p \pr_* \left( R^n \cap [Z] \right) \\
%               &= (2 \pi i)^{p-n} \begin{cases} R_0^n \cap [Z], & n = 2 p  \\ R_1^n \cap [Z], & n = 2p - 1 \end{cases}
%\end{align*}
\begin{align*}
    r_P(Z) &= (2 \pi i)^p \pr_* \left( R^n \cap [ Z ] \right) 
\end{align*}
in the first case is computed by the summand $R_0^n$ of $R^n$ and in the second case by $R^n_1$. 
In both cases the intersection is seen to be zero dimensional, i.e., a sum of points (in the second case tensored with $dx$). Thus the push forward is just the sum of the coefficients of these points, multiplied with the relative dimension $(2 \pi i)^{-n}$. 

In the first case only two types of currents occur. One can use shuffles to sort them and write
\begin{align*}
   R_0^n \cap [ Z ] 
     &= \sum_i \sum_{\sigma \in \Sh(i,n-i)} (-1)^{|\sigma|} (2\pi i)^{i-p} (1-x)^i x^{n-i} \sigma_* \left( {\Rm}^{\boxtimes i} \boxtimes \dlog^{\boxtimes n-i} \right) \cap [Z].     
\end{align*}
Note that the summation index starts at $i=p$ since otherwise the resulting current lies in the $n-i+p > n$-th part of the Hodge filtration and thus is zero. 

In the second case in addition a single $\log$ term shows up. An explicit formula is complicated to write down\footnote{It is something like 

$\sum \limits_{\genfrac{}{}{0pt}{1}{i = 0...n-1,}{j=1..i+1 \ \ }} \sum \limits_{\sigma \in \Sh(i,n-i)} (-1)^{|\sigma|+\epsilon} (2\pi i)^{i-p+1} (1-x)^i x^{n-i} dx \otimes \sigma_* \left( \Rm \boxtimes \ldots \underset{j}{\log} \ldots \boxtimes \Rm \boxtimes \dlog \boxtimes \ldots \boxtimes \dlog \right) \cap [Z]$
where in the sum the currents $\Rm$ and $\dlog$ occur $i$ resp. $n-i-1$ times.}, so we explain the result informally. It is, basically, the intersection of $Z$ with all possible currents of the form
%$\sigma_* \Big( \log(z_1) \dlog (z_2) \wedge \ldots \wedge \dlog(z_i) \times (\Rm)^{n-i} \Big)$ 
$\sigma_* \left( \dlog \boxtimes \ldots \underset{j}{\log} \ldots \boxtimes \dlog \boxtimes \Rm \boxtimes \ldots \boxtimes \Rm \right)$, with $\sigma$ shuffling the forms and the $\Rm$ together.

The Deligne cohomology of a point is easy to describe since the 3 term complex is simply 
\begin{align*}
     C_\calD^\bullet(\pt, \Z(p)) 
     &= ( \Z(p) \oplus F^p \C \to \C )\\
     &\cong \begin{cases} \Z   & p = 0, \\ 
                          \C / \Z [-1]  & p > 0. 
            \end{cases}
\end{align*}
%
%The Deligne cohomology of a point is easy to describe and in fact it can directly be computed as 
%\[
%     H_\calD^\bullet (\pt,\Z(p)) \cong  \begin{cases} ( \Z \to 0 ) & p = 0 \\ ( 0 \to \C/ \Z(p))  & p > 0. \end{cases}
%\]
%%The regulator into Deligne cohomology can, on the level of chains, be computed by 
The regulator map into this complex can be expressed on the level of chains as the map 
\begin{align*}
\calN_\R^\bullet(\pt,p) = z^p_\R(\pt, 2p-\bullet) \xrightarrow{r_P} P^{\bullet} \left(\pt, \Z(p) \right) \xrightarrow{ev} \begin{cases} \Z  & p = 0, \\ (\Z(p) \to \C)  & p > 0. \end{cases}
\end{align*}
%\begin{align*}
%z_\R^p(\pt,\bullet) \xrightarrow{r_\calP} P^{2p-\bullet} \left(\pt, \Z(p) \right) \xrightarrow{ev} \begin{cases} \Z , & p = 0 \\ (\Z(p) \to \C) , & p > 0. \end{cases}
%\end{align*}
This map is determined by sending for $p = \bullet = 0$ an element $k \cdot [\pt]$ to $k \in \Z$, and for $p>0$ sending a cycle $Z$ to $\left( r_P(Z)|_{x = 0}, \int_0^1 r_P(Z) \right)$. 
In each other case, the map is zero.
Passing to cohomology groups and taking into account that $\calN_\R^\bullet$ computes motivic cohomology, one obtains the induced regulator 
\[
    r_P : H_M^n(\pt,\Z(p))    \longrightarrow  \begin{cases} \Z & p=0, n=0, \\ 
                                                             \C / \Z(p) & p > 0, n = 2p-1, \\
                                                             0  & \text{else. }
                                              \end{cases} 
\]
given by 
\[
\begin{array}{cccl}
    k \cdot \pt & \mapsto & k                 & \quad p=0, n=0, \\
    Z           & \mapsto & \int_0^1 r_P(Z)   & \quad p > 0, n=2p-1. 
\end{array} 
\]

For $n=p=1$ one obtains (see the formula in the next subsection) the logarithm map 
\[
   r_P = \log : H^1_M(\pt,\Z(1)) = \C^* \to \C/2 \pi i \Z.
\]
%\[
%    r : \CH^p_\R(\pt, n)    \xrightarrow{}  \begin{cases} \Z & p=0 \\ (0 \to \C / \Z(p)), & p > 0 \end{cases} 
%\]
%is given by
%\[
%\begin{array}{cccl}
%    k \cdot \pt & \mapsto & k,                                 & \quad p=0, n=0 \\
%    Z           & \mapsto & \left( 0, \int_0^1 r(Z) \right),   & \quad p > 0, n=2p-1 \\
%    Z           & \mapsto & 0,                                 & \quad \text{else. }
%\end{array} 
%\]
% 
% 

\subsection{Formulas for $n \leq 3$}
We now calculate some low-dimensional examples of the regulator map and their images in the 3-term-complex under the evaluation map $ev$.
They are easily read off from the computation of $R^n$ in section \ref{sec:regulator_to_holim}. 
\begin{itemize}
\item 
First of all, consider the case where $n = 0$. Here $r_P$ reduces to the cycle map from the usual Chow groups to the Deligne complex $P_\calD$. Composition with $ev$ gives the cycle map into $C_\calD$.

\item
For $n=1$ the regulator $r_P(Z)$ is the push-forward to $X$ of 
%\[
%     (1-x) (2 \pi i)^p [X \times \R^-] \cdot [Z] + x (2 \pi i)^{p-1} [ \dlog z ] \cdot [Z] + dx (2 \pi i)^{p-1} [ \log z ] \cdot [Z].
%\]
\[
     (1-x) (2 \pi i)^{p+1} [U \times \Rm] \cap [Z] + x (2 \pi i)^{p} [ \dlog z ] \cap [Z] + dx (2 \pi i)^{p} [ \log z ] \cap [Z].
\]
After writing the push forward as a fiber integral over $X$ twisted by $(2\pi i)^{-1}$, this becomes (with $z$ the coordinate in $\cube$)
\[
     (1-x) (2 \pi i)^p \int_{(U \times \Rm) \cap Z} + x (2 \pi i)^{p-1} \int_{Z}  \dlog z + dx (2 \pi i)^{p-1} \int_{Z} \log z.
\]
In all these formulas one has to exclude the poles of the integrand from the integration domain, that is the two last integrals are over $Z_\reg \setminus (U {\times} \{0,\infty\})$ and $Z_\reg \setminus (U {\times} \Rm)$ respectively.

In particular, the regulator into $P_\calD$ contains for $n=1$ no new information compared to the regulator in the 3-term complex, which is given by the triple of currents on $X$
\[
     (2\pi i)^{p-1} \left( 2 \pi i \int_{(X \times \Rm) \cap Z} , \int_{Z} \dlog z , \int_{Z} \log z \right).
\]

Important\footnote{Using the Gersten resolution for higher Chow groups \cite[§10]{BlochAlgebraicCycles}, the degeneration of the local to global spectral sequence for higher Chow groups \cite[§5]{MuellerStach_AlgebraicCycleComplexes}, and the Milnor-Chow homomorphism \cite{Totaro}, one can see that in fact each class in $\CH^p(U,1)$ can be represented by a sum of such graphs.}  cycles in $z^p(U,1)$ are the cycles represented by graphs $\Gamma_f$, where $f$ is a non-zero rational function, defined on a codimension $p-1$ algebraic subvariety $V \subset U$. 
%
%The higher Chow group $CH^p(U,1)$ is generated\footnote{More precisely, one need sums of such tuple whose sum of divisors is zero} by (the graphs of) tuples $(V,f)$ where $V$ is a codimension $p-1$ algebraic subvariety and $f$ is a non-zero rational function on $V$. 
The regulator of such an element is 
\[
    r_C( \Gamma_f ) = (2\pi i)^{p-1} \left( 2 \pi i [f^{-1}(\Rm)] , \int_{V} \dlog f , \int_{V} \log f \right).
\]

\item
For $n=2$ one gets a symmetrization of (a $\C$ version of) a formula found by Beilinson in \cite{BeilinsonHigherRegulatorsAndValuesOfLFunctionsOfCurves}. 
The regulator value of $Z \in z^p(X,2)$ in the complex $P_\calD$ is: 
\begin{align*}
   &- (2 \pi i)^{p} (1-x)^2 \int_{Z \cap (X \times \Rm \times \Rm)} 
    + (2 \pi i)^{p-2} x^2 \int_Z \dlog z_1 \wedge \dlog z_2 \\
   &+ (2 \pi i)^{p-1} x (1-x) \left( \int_{Z \cap (X \times \Rm \times \cubeBar)} \dlog z_2 - \int_{Z \cap (X \times \cubeBar \times \Rm)} \dlog z_1 \right) \\
   &+ (2 \pi i)^{p-2} x dx \left( \int_Z \log z_1 \dlog z_2 - \int_Z \log z_2 \dlog z_1 \right) \\
   &+ (2 \pi i)^{p-1} (1-x) dx \left( \int_{Z\cap(X \times \cubeBar \times \Rm)} \log z_1 - \int_{Z \cap (X \times \Rm \times \cubeBar)} \log z_2 \right).   
\end{align*}
In the total complex its first two components are the currents
\begin{equation*}
   -(2 \pi i)^{p} \int_{(X \times \Rm \times \Rm) \cap Z}, \qquad \qquad 
    (2 \pi i)^{p-2} \int_{Z} \dlog z_1 \wedge \dlog z_2 
\end{equation*}
and its third component is 
\begin{equation*}
    \frac{(2 \pi i)^{p-2} }{2} \left( \int_{Z} \log z_1 \dlog z_2 - \log z_2 \dlog z_1 + 2 \pi i \int_{ Z \cap \{ z_2 \in \Rm\}  }  \log z_1 - 2 \pi i \int_{Z \cap \{ z_1 \in \Rm\}} \log z_2 \right). 
\end{equation*}
The regulator value of $Z$ in the $3$ term complex on the other side is 
\[
(2 \pi i)^{p-2} \Big( - (2 \pi i)^2 \int_{Z \cap \{z_1,z_2 \in \Rm\}},  \int_Z \dlog z_1 \wedge \dlog z_2, \int_Z \log z_1 \dlog z_2 - 2 \pi i \int_{Z \cap \{ z_1 \in \Rm\}} \log z_2 \Big).
\]
\item
Examining the case $n=3$, we consider only the evaluation of the regulator values in the total complex. This is 
%\[
% (2 \pi i)^{p-3} \left( (2 \pi i)^3 [\Rm \times \Rm \times \Rm] , [\dlog] \times [\dlog] \times [\dlog] , R_Z \right)
%\]
\[
 (2 \pi i)^{p-3} \Big( -(2 \pi i)^3 \int_{Z \cap (X \times (\Rm)^3)} 
                       \,,\, \int_{Z} \dlog z_1 \wedge \dlog z_2 \wedge \dlog z_3 
                       \,,\, R_Z 
                 \Big)
\]
where $R_Z$ is the current given by %(using that the term $\Rm$ should be at the tail to allow integration)
%\begin{align*}
%-\scriptfrac{1}{3} \int_{Z_\reg} \log z_1 \dlog z_2 \dlog z_3 - \log z_2 \dlog z_1  \dlog z_3 + \log z_3 \dlog z_1 \dlog z_2 \\
%-\scriptfrac{2 \pi i}{6} \int_{Z_\reg \cap \{ z_1 \in \R^-\}} \log z_2 \dlog z_3 - \dlog z_2 \log z_3 \\
%-\scriptfrac{2 \pi i}{6} \int_{Z_\reg \cap \{ z_2 \in \R^-\}} \dlog z_1 \log z_3 - \log z_1 \dlog z_3 \\
%-\scriptfrac{2 \pi i}{6} \int_{Z_\reg \cap \{ z_3 \in \R^-\}} \log z_1 \dlog z_2 - \dlog z_1 \log z_2 \\
%-\scriptfrac{(2 \pi i)^2}{3} [\log \Rm \Rm - \Rm \log \Rm + \Rm \Rm \log ] \cdot [Z]
%\end{align*}
\begin{align*}
\scriptfrac{1}{3} \int_{Z} \log z_1 \dlog z_2 \dlog z_3 - \log z_2 \dlog z_1  \dlog z_3 + \log z_3 \dlog z_1 \dlog z_2 \\
+\scriptfrac{2 \pi i}{6} \int_{Z \cap \{ z_1 \in \Rm\}} \log z_3 \dlog z_2 - \log z_2 \dlog z_3 \\
+\scriptfrac{2 \pi i}{6} \int_{Z \cap \{ z_2 \in \Rm\}} \log z_1 \dlog z_3 - \log z_3 \dlog z_1 \\
+\scriptfrac{2 \pi i}{6} \int_{Z \cap \{ z_3 \in \Rm\}} \log z_2 \dlog z_1 - \log z_1 \dlog z_2 \\
-\scriptfrac{(2 \pi i)^2}{3} \Big[ \int_{Z \cap \{ z_1, z_2 \in \Rm\}} \log(z_3) - \int_{Z \cap \{ z_1, z_3 \in \Rm \}} \log(z_2) + \int_{Z \cap \{ z_2, z_3 \in \Rm\}} \log(z_1) \Big].
\end{align*}
%%%% Alternative für die letzte Zeile
%+\scriptfrac{(2 \pi i)^2}{3} [\log \Rm \Rm - \Rm \log \Rm + \Rm \Rm \log ] \cap [Z].
To make this more concrete, one can apply this formula to the cycle $C(1)$ considered by Burt Totaro in \cite[\textsection 2]{Totaro}, which by definition is the algebraic cycle in $\cube^3$ parametrized by 
\[ 
   \varphi(t) = (t,1-\scriptfrac{1}{t},1-t), \qquad t \in \P^1 \setminus \{ 0,1,\infty\}.
\]
One gets that the first and the last row vanishes, and each other term becomes $\pi^2/6$. Thus
\begin{align*}
   \int r_P(C(1)) = \scriptfrac{\pi^2}{6} = \Li_2(1)
\end{align*}
is a special value of the dilogarithm function. 
\end{itemize}

\begin{remark*}
The importance of $C(1)$ lies in the fact, that $\frac{\pi^2}{6}$ has order $24$ in $\C/\Z(2) = \C/4 \pi \Z$ and hence $C(1)$ is a 24-torsion element in $\CH^2(\pt,3)$.
\end{remark*}

\subsection{The general Totaro cycle in $\cube^3$}
We continue the above example of the Totaro cycle and consider, following \cite{KLM_AbelJacobi}, more general for a parameter $ a \in \P^1(\C) \setminus (\R_{\leq 0} \cup \R_{\geq 1})$ the cycles 
\[
     C(a) = \left\{ (z,1-\scriptfrac{a}{z},1-z) :  z \in \P^1(\C) \right\} \cap \cube^3.
\]
They have Bloch boundary $\partial C(a) = - (a,1-a)$, and thus are higher Chow cycles if and only if $a \in \{0,1\}$.
Nevertheless, for $a$ as above, the $C(a)$ intersect the real boundaries properly and we can calculate the regulator for such $a$. We use the formulas from above and see that, noting that the first (by dimensionality) and the last (by choice of $a$) row vanishes, we have to compute
\begin{align*}
  A &= \int_{-\infty}^0 \log (1-z) \dlog (1-a/z) - \int_{-\infty}^0 \log(1-a/z) \dlog (1-z) \\
  B &= \int_{0}^a \log z \dlog (1-z) - \int_0^a \log (1-z) \dlog z \\
  C &= \int_{\infty}^1 \log (1-a/z) \dlog z - \int_{\infty}^1 \log z \dlog (1-a/z).  
\end{align*}
Integration by parts reduces each of the above pairs of integrals to a single integral and a limit. Evaluation with mathematica gives 
\begin{align*}
  A &= 2 \int_{-\infty}^0 \log (1-z) \dlog (1-\scriptfrac{a}{z}) 
       - \lim_{h \nearrow 0} \Big[ \log(1-\scriptfrac{a}{h}) \log (1-h) - \log(1-a h) \log (1-\scriptfrac{1}{h}) \Big] \\
   &= 2 Li_2(a) + 2 \log(a) \log(1-a).
\intertext{Similar for $B$,$C$:}
  B &= -2 \int_0^a \log (1-z) \dlog z + \lim_{h \nearrow 1} \left[ \log(a h) \log(1-a h)) \right] - \lim_{h \searrow 0} \left[ \log(h) \log(1-h) \right] \\
    &= 2 Li_2(a) + \log(a) \log(1-a) \\[20pt]
  C &= 2 \int_{\infty}^1 \log (1-a/z) \dlog z  - \lim_{h \nearrow \infty} \left[ \log(1-\scriptfrac{a}{1-\scriptfrac{1}{h}}) \log(1-\scriptfrac{1}{h}) - \log(1-\scriptfrac{a}{h}) \log(h) \right] \\
    &= 2 Li_2(a).
\end{align*} %\scriptfrac{a}{1-\scriptfrac{1}{h}}  %a/(1-\scriptfrac{1}{h}
Thus the regulator in this case is 
\begin{align*}
   r_P (C(a)) 
     &= (2\pi i)^{2-3} \cdot (2\pi i) x (1-x) dx (A + B + C)  \\
     &= x (1-x) dx \big( 6 Li_2(a) + 3 \log(a) \log(1-a) \big).
\end{align*}
In the 3-term complex, this becomes $(0,0,Li_2(a) + \scriptfrac{1}{2} \log(a) \log(1-a))$. 
For $a \to 1$ this reduces to the value already computed. 
The KLM regulator on the other side, when applied to the above cycle, can easily computed to be 
\[
    r_{C,\cap_0} (C(a)) = (0,0,Li_2(a) + \log(a) \log(1-a)).
\]
In particular, the two regulators are not equal. 

\begin{remarks}
%\item In this example all regulator values are equal ($=Li_2(a)$) up to polylogarithms of lower order. 
\item Following KLM \cite{KLM_AbelJacobi}, one can also consider the curve $D(b) := \{(1-z,1-b/z,z) \} \cap \cube^3$. 
      Then $C(a) - D(1-a)$ is a higher Chow cycle. 
      Since $D(b)$ is obtained from $C(b)$ by exchanging the first two components, the resulting regulator is $r_P(D(b)) = - r_P(C(b))$. 
      Thus we have, using the transformation rules for the dilogarithm \citep[(3.3)]{Maximon_complex_dilogarithm},
      \begin{align*}
         r_P(C(a)-D(1-a)) 
           &= r_P(C(a)) + r_P (D(1-a)) \\
           &=  6 x (1-x) dx Li_2(1).
      \end{align*}
      This means that the error terms cancel. 
      The same holds for the regulator in the total complex and in fact, one has $ev (r_P(C(a)-D(1-a)) )= r_{C,\cap_0}(C(a)-D(1-a))$.
\end{remarks}

\section{The Abel-Jacobi map}
%This section shows how the regulator $r_P$ gives rise to an Abel-Jacobi map from the group of higher Chow cycles homologous to zero to some intermediate Jacobian. This construction is analogous to the one in \cite{KLM_AbelJacobi} for the regulator $r_C$.
%
%We assume that $U = X$ is projective, that is, $D=0$. Furthermore we omit $D$ from the notation and write $P_\calD(X,\Z(p))$ for $P_\calD(X,D,\Z(p))$ etc. 
%
%\paragraph{Cycles homologous to zero}
%Following \cite{KLM_AbelJacobi}, we say that a higher Chow chain $Z \in z^p_\R(X,n)$ is homologous to zero, if $\pr \circ r_C(Z)$ is a boundary in $\calI^{2p-n}(X,\Z(p)) \oplus F^p \calD^{2p-n}(X,\C)$, where $\pr$ denotes the projection from $C_\calD$ onto it's first two components. 
%%
%By lemma \ref{lemma:comparison_regP_and_regC}, this is equivalent to saying that $Z$ lie in the kernel of the composition $\pr \circ ev \circ r_P$.
%
%
%
%Written as a diagram, the chains homologous to zero are exactly those chains that become boundaries in the rightmost term of the diagram below
%\[ 
%\xymatrix{
%    & P_\calD^{2p-n}(X,\Z(p)) \ar[d]^{ev} \\
%    z^p_\R(X,n) \ar[r]^{r_C} \ar[ru]^{r_P}  & \  C_\calD^{2p-n}(X,\Z(p))  \ar[r] & \calI^{2p-n}(X,\Z(p)) \oplus F^p \calD^{2p-n}(X,\C). 
%}
%\]
%
%
This section shows how the regulator $r_P$ gives rise to an Abel-Jacobi map from the group of higher Chow cycles homologous to zero to some intermediate Jacobian. This construction is analogous to the one in \cite{KLM_AbelJacobi} for the regulator $r_C$.

We assume that $U = X$ is projective, that is, $D=0$. Furthermore we omit $D$ from the notation and write $P_\calD(X,\Z(p))$ for $P_\calD(X,D,\Z(p))$ etc.

\paragraph{Cycles homologous to zero}
Following \cite{KLM_AbelJacobi}, we say that a higher Chow chain $Z \in z^p_\R(X,n)$ is homologous to zero, if $\proj \circ r_C(Z)$ is a boundary in $\calI^{2p-n}(X,\Z(p)) \oplus F^p \calD^{2p-n}(X,\C)$, where $\proj$ denotes the projection from $C_\calD$ onto it's first two components. 
By lemma \ref{lemma:comparison_regP_and_regC}, this is equivalent to saying that $Z$ lie in the kernel of the composition $\proj \circ ev \circ r_P$.

Written as a diagram, the chains homologous to zero are exactly those chains that become boundaries in the rightmost term of the diagram below
\[ 
\xymatrix{
    & P_\calD^{2p-n}(X,\Z(p)) \ar[d]^{ev} \\
    z^p_\R(X,n) \ar[r]^-{r_C} \ar[ru]^{r_P}  & \  C_\calD^{2p-n}(X,\Z(p))  \ar[r]^-{\proj} & \calI^{2p-n}(X,\Z(p)) \oplus F^p \calD^{2p-n}(X,\C). 
}
\]

The set of all higher Chow chains homologous to zero form a subgroup of the higher Chow chains, denoted by 
\[
    z^p_{\R,\hom}(X,n).
\]

If $Z$ is a boundary, then $Z$ is automatically homologous to zero. Hence this notion makes sense on cohomology classes and we define $\CH_{\hom}^p(X,n)$ to be the cohomology classes represented by cycles homologous to zero. This is the kernel of the composition 
%\[
%     \CH^p(X,n) \to H^{2p-n} P_\calD^\bullet(X,\Z(p)) \to H^{2p-n} C_\calD^\bullet(X,\Z(p)) \to  H^{2p-n}(X,\Z(p)) \oplus F^p H^{2p-n}(X,\C).
%\]
%Say that a higher Chow cycle $Z \in z_\R^p(X,n)$ is homologous to zero, if it lies in the kernel of the composition 
\[ 
    \CH^p(X,n) \to H_\calD^{2p-n}(X,\Z(p)) \to  H^{2p-n}(X,\Z(p)) \oplus F^p H^{2p-n}(X,\C)
\]
that is induced by the regulator into the total complex followed by the projection onto the first two components.

The notion of being ''homologous to zero'' depends on the regulator map. The following lemma gives a criteria for checking this property (for $\partial_B$ closed cycles) without explicitly mentioning the regulator.

\begin{lemma} % \label{lemma:homologous_to_zero}
Let $Z \in z^p_\R(X,n)$ be an higher Chow cycle, i.e. $\partial_B Z = 0$. Then 
\begin{center}
$Z$ is homologous to zero if and only if $(\pr_X)_* ([X \times (\Rm)^{\times n}] \cap [Z])$ is zero in $H^{2p-n}(X,\Z)$.
\end{center} 
\end{lemma}
\begin{proof}
Write the regulator of $Z$ as 
\[
     r_C(Z) = (T_0,T_1, \text{rest}).
\]
Note that the current in the statement of the lemma is just $\pm (2 \pi i)^p T_0$. 
Thus it is to show that 
\[
    T_0 = 0 \text{ in } H^{2p-n}(X,\Z(p)) \iff \begin{cases} T_0 = 0 \text{ in } H^{2p-n}(X,\Z(p)) \\ T_1 = 0 \text{ in } F^p H^{2p-n}(X,\C) \end{cases} 
\]
It suffices to show that if $T_0$ is a boundary in cohomology with $\Z(p)$ coefficients, then $T_1 = 0$ in $F^p H^{2p-n}(X,\C)$. 
Since $Z$ is assumed to be $\partial_B$-closed, $T_1 = T_0 + d( \text{rest} )$ in $H^{2p-n}(X,\C)$. 
If $T_0$ is a boundary in $\Z(p)$-valued cohomology, then also with $\C$ coefficients. Thus $T_1$ is a boundary in $H^{2p-n}(X,\C)$. By the $d' d''$-lemma \cite[1.2.1]{GilletSoule_Arithmetic_Intersection_Theory} the bounding current can be choosen in $F^p \calD^{2p-n-1}(X)$ and hence $T_1$ is also a boundary in $F^pH^{2p-n}(X,\C)$.
%But $T_1$ lies in the image of the map
%\[
%   F^p H^{2p-n}(X,\C) \to H^{2p-n}(X,\C) 
%\]
%which is injective by ...... In particular, $T_1$ is also a boundary in $F^pH^{2p-n}(X,\C)$.
\end{proof}

\paragraph{The Abel-Jacobi map}
Every total complex gives rise to a long exact sequence on cohomology. In particular, there is a long exact sequence associated to the total complex $C_\calD(X,\Z(p))$. From this long exact sequence one can extract the exact sequence 
\[
     0 \to J^{p,n}(X) \to H^{2p-n}C_\calD^\bullet(X,\Z(p)) \xrightarrow{\proj} H(X,\Z(p)) \oplus F^p H(X,\C)
\]
where 
\[
 J^{p,n}(X) := \frac{H^{2p-n-1}(X,\C)}{H^{2p-n-1}(X,\Z(p)) + F^p H^{2p-n-1}(X,\C)}
\]
and the map into the 3-term complex is induced by $T \mapsto (0,0,T)$.
For any cycle homologous to zero, the regulator values $r_C(Z)$ and $ev \circ r_C(Z)$ both lie in the kernel of the projection, hence come from an element in $J^{p,n}$. 
This uniquely defines the Abel-Jacobi maps with respect to the regulators $r_P$ and $r_C$. The construction is summarized by the diagram 
%\[
%\xymatrix{
%    z^p_{\hom}(X,n) \ar[d]^{AJ_P} \ar[r]^{r_P}  &   H^{2p-n} P_\calD(X,\Z(p)) \ar[d]^{ev} \\ 
%    J^{p,n}(X) \ \ar@{^(->}[r]^-{i} & H_\calD^{2p-n}(X,\Z(p))
%}
%\]
\[
\xymatrix{
   \CH^p_{\hom}(X,n) \ \ar@{^(->}[r] \ar@<2pt>[dd]^{AJ_P} \ar@<-2pt>[dd]_{AJ_C} & \CH^p(X,n) \ar[dd]^{r_C} \ar[rd]^{r_P} \\
   & & H^{2p-n}P_\calD(X,\Z(p)) \ar[ld]_{ev} \\
   J^{p,n}(X) \ \ar@{^(->}[r] & H^{2p-n}C_\calD(X,\Z(p)) 
}
\]
Note that although their projection onto the first two components $\proj \circ r_C(Z) = \proj \circ r_P(Z)$ is the same, the values of $r_C(Z)$ and $ev \circ r_P(Z)$ in general are different, hence give rise to different Abel-Jacobi maps. 

\paragraph*{Explicit formulas}
The Abel-Jacobi map can be made explicit after observing that the image of $J^{p,n}(X)$ consist of those triples where at most the third component is $\neq 0$. 

Thus the general procedure to construct the Abel-Jacobi map from a regulator value in the 3-term complex is to: First use that $Z$ is homologous to zero to move all information into the third component by adding a boundary. Then project to the third component and take a quotient to be independent of the choices made for the boundary. 

In case of the regulator $r_C$, write the regulator value of a higher Chow chain $Z$ as 
\[
     r_C(Z) = (T_0, T_1, r_C(Z)_3 ). 
\]
If $Z$ is homologous to zero, $T_0 = d S_0$ and $T_1 = d S_1$ for some $S_0 \in \calI^{2p-n-1}(X,\Z(p))$ and $S_1 \in F^p \calD^{2p-n-1}(X)$. 
Now the regulator value is equivalent to 
\begin{align*}
    r_C(Z) =  (T_0, T_1, r_C(Z)_3 ) - d (S_0, S_1, 0) = (0,0, r_C(Z)_3 - S_1 + S_0).
\end{align*} 
The resulting Abel-Jacobi map is given by 
\[
     AJ_C(Z) := r_C(Z)_3 + S_0 - S_1.
\] 
%where $r_C(Z)_3$ is the third component the regulator value and $S_0$,$S_1$ bound the first resp. second components of the regulator. 

Similarly, write 
\[
    ev \circ r_P(Z) = (T_0, T_1, \int_0^1 r_P(Z)_1 ). 
\]
Note that by lemma \ref{lemma:comparison_regP_and_regC}, $T_0,T_1$ are exactly the equally named currents as above. Thus we can choose the same boundaries $S_0,S_1$ as before and obtain that the Abel-Jacobi map with respect to the regulator $r_P$ can be described by the formula
\begin{align} \label{eq:abel_jacobi_into_P}
     AJ_P(Z) :=  \int_0^1 r_P(Z)_1 + S_0 - S_1.
\end{align}

\begin{remark}  \label{rmk:AJ_for_n_greater_p}
For $n \geq p$ one has $F^p \calD^{2p-n-1}(X) = 0$. Hence the Abel-Jacobi map lands in $\frac{H^{2p-n-1}(X,\C)}{H^{2p-n-1}(X,\Z(p))}$, computed by the formula \eqref{eq:abel_jacobi_into_P} with $S_n$ set to zero. 
\end{remark}

The lemma \ref{lemma:comparison_regP_and_regC} leads to the following comparison of the two Abel-Jacobi mappings. 
\begin{corollary}
One has 
\[ 
      AJ_P = AJ_C \circ \alt
\]
as maps $CH^p_{\hom}(X,n) \to \frac{ H^{2p-n-1}(X,\C) }{ H^{2p-n-1}(X,\Z(p)) +  F^p H^{2p-n-1}(X,\C)}$. 
\end{corollary}

\begin{proof}
The statement is equivalent to $\int_0^1 r_P(Z)_1 = r_C(\alt Z)_3$, which follows from lemma  \ref{lemma:comparison_regP_and_regC} resp. the remark following it. 
\end{proof}

\paragraph*{Examples}

\begin{itemize}
\item For $n=0$ the regulator is just the cycle map and the cycles homologous to zero are
\[
    \CH^p_{\hom}(X,0) = \left\{ Z \in \CH^p(X) : \cl(Z) = 0 \text{ in } H^{2p}(X,\Z(p)) \right\}.
\]
The Abel-Jacobi map is $AJ_P(Z) = S_0 - S_1$, where $S_0$, $S_1$ are currents bounding $\cl(Z)$ such that $S_0$ is integral with coefficients in $\Z(p)$ and $S_1$ lies in $F^p \calD^{2p-1}$. 

Note that the Poincaré duality induces an isomorphism $J^{p,0}(X) \cong \frac{F^{m-p+1} H^{2m-2p+1}(X,\C)^\vee}{H^{2m-2p+1}(X,\Z(p))^\vee}$ and the image of $AJ(Z)$ under this isomorphism is given by $S_0$ only ($S_1$ acts trivial). Finally, writing $S_0 = [\xi]$ as the current of integration over a singular chain $\xi$, one reobtains the classical Abel-Jacobi map of Griffiths.

\item 
For $n = 1$ and any $Z \in z^p_\R(X,1)$ homologous to zero, its Abel-Jacobi value is the current on $X$ 
\[
     AJ(Z) = (2 \pi i)^{p-1} ( \int_Z \log(z) + 2 \pi i S_0 - S_1).
\] 

The currents $S_i$ satisfy $d S_0 = [Z \cap X {\times} \Rm]$ and $d S_1 = [Z] \wedge \dlog(z)$. 

\item 
In particular, if $Z = \sum \Gamma_\alpha$ is the sum of graphs of meromorphic functions $f_\alpha : V_\alpha \to \P^1$, this recovers a version of Levine's formula (\cite[p. 458]{Levine_LocalizationOnSingularVarieties} and \cite[4.5]{KLM_AbelJacobi}): By remark \ref{rmk:AJ_for_n_greater_p}, $S_1$ acts trivially in this case and thus, for $\xi$ any cycle with boundary $Z \cap (X {\times} \Rm)$, 
\[ 
AJ(Z) = (2 \pi i)^{p-1} ( \sum_\alpha \int_{V_\alpha \setminus f_\alpha^{-1} \Rm} \log f_\alpha + 2 \pi i \int_\xi ).
\]

\item Totaro's cycle $C(1)$ is homologous to zero. $T_0,T_1$ are actually zero (not only boundaries) so that one can choose $S_0 = S_1 = 0$. It's Abel-Jacobi image is thus 
\[
     AJ_P(C(1)) = \int r_P(Z) = \Li_2(1).
\]
\end{itemize}

\newpage
\appendix

\section{Requirements for the Deligne complexes} \label{appendix:requirements_on_D_complexes}
Denote by $(X,D)$ a pair consisting of a smooth projective variety and a normal crossing divisor $D \subset X$. Assume that to every such a pair and every integer $p$ (the weight) there is associated a (Deligne) complex of $\Z$ modules $C_\calD^\bullet(X,D,\Z(p))$. 
We summarize all the properties of these complexes that are needed to apply the general construction from section \ref{sec:abstract_regulator}.
\begin{itemize}
\item (Cycle map) There exist maps $\cl : Z^p(X \setminus D) \to C_\calD^{2p}(X,D,\Z(p))$ such that $d \circ \cl = 0$. 
\item (Flat pullback) For any holomorphic submersion $Y \xrightarrow{f} X$ one has a morphism of complexes  
      \[ 
          f^* : C_\calD^\bullet (X,D,\Z(p)) \to C_\calD^\bullet (Y,f^{-1}D,\Z(p)).
      \]
\item (Proper push forward) For any smooth proper morphism of pairs $(X,D) \xrightarrow{f} (X',D')$ one has a morphism of complexes  
      \[
           f_* : C_\calD^\bullet(X,D,\Z(p)) \to C_\calD^{\bullet-2\delta}(X',D',\Z(p-\delta))
      \]
      where $\delta = \dim_\C X - \dim_\C X'$ is the relative dimension. 
      In particular, $d f_* = f_* d$.
\item (Products) There exist external products 
      \[ 
             C_\calD^\bullet (X,D,\Z(p)) \otimes C_\calD^\bullet (X',D',\Z(q)) \xrightarrow{\boxtimes} C_\calD^\bullet (X \times X', D \boxtimes D',\Z(p+q))
      \]
      in the category of complexes. 
      Here $D \boxtimes D' = X \times D' + D \times X'$ denotes the exterior product of the divisors. 
      Furthermore there exist a partially defined (internal) product 
      \[ 
             C_\calD^\bullet(X,D,\Z(p)) \otimes C_\calD^\bullet(X,D,\Z(q)) \xrightarrow{\cap} C_\calD^\bullet(X,D,\Z(p+q)).
      \]
\end{itemize}

%There are a couple of compatibility conditions we expect to hold for the above structures. They are listed below:
The above structures should satisfy some compatibility conditions listed below:
\begin{itemize}
\item $\cl$ is compatible with flat pullback and push forward, that is, 
      \[
            \cl \circ f|_{X \setminus D}^* = f^* \circ \cl 
            \quad \text{ resp. }     \quad
            \cl \circ (f|_{X\setminus D})_* = f_* \circ \cl
      \]
      whenever $f|_{X \setminus D}$ is flat resp. proper. 
\item $\cl(X \setminus D)$ is an unit for $\cap$ and $\cl(\pt)$ is an unit for $\times$.
\item Push forward along the diagonal is injective.
\item Inverse mapping formula holds: For $f : X \to X'$ holomorphic, one has an equation 
      \[
           S \cap f_* T = f_* ( f^* S \cap T )
      \]
      whenever the right hand side is defined. 
\item Reduction to the diagonal: If $\Delta : X \to X \times X$ denotes the diagonal and $\Delta_X$ its image, then 
      \[
           \Delta_* ( S \cap T ) = ( S \boxtimes T ) \cap \cl(\Delta_X).
      \]
\item Compatibility of $\boxtimes$ and $\cap$ for cycles coming from geometry (i.e. of type $\cl(Z)$ for an algebraic variety $Z$): If the intersection on the right hand side exist, there is an equality 
%\item Compatibility of $\boxtimes$ and $\cap$, whenever one side is formed by cycles coming from geometry (i.e. are of type $\cl(Z)$ for an algebraic variety $Z$):
      \[
           (S_1 \boxtimes S_2) \cap (\cl(Z_1) \boxtimes \cl(Z_2)) = (S_1 \cap \cl(Z_1)) \boxtimes (S_2 \cap \cl(Z_2)).
      \]
\item $\boxtimes$ and $\cap$ are associative, $\boxtimes$ is graded-commutative.
\item $\cap$ is preserved by analytic isomorphisms: $\varphi_* S \cap \varphi_* T = \varphi_* (S \cap T)$, if $\varphi$ is orientation preserving.
\item Push forward preserves $\boxtimes$.
\end{itemize}

\section{Currents} \label{appendix:currents}
Here we briefly summarize the needed theory of currents on complex manifolds and their intersection. 
A current on a complex manifold $X$ is a continuous linear functional on the complex vector space of compactly supported smooth differential forms on $X$. 
The currents on $X$ form a complex vector space denoted by $\calD(X)$. 
This space is naturally equipped with a bigraduation by duality with forms where $\calD^{r,s}(X)$ is the set of currents that vanish on all test forms of bidegree $\neq (m-r,m-s)$, $m = \dim_\C X$. 
The associated single-graded vector spaces $\calD^n(X) := \oplus_{r+s = n}\calD^{r,s}(X)$ becomes a (cochain) complex via the differential of degree $+1$ that sends a current $T$ to the current $dT$ defined by 
\[
     d T(\eta) = (-1)^{\deg(T)+1} T(d \eta).
\]
For a pair $(X,D)$ as considered above, King \cite{King_LogComplexesOfCurrents} defined the quotient complex $\calD(X,\log D) := \calD(X) / \calD(X, on D)$ of $\log D$ currents, where $\calD(X , on D) \subset \calD(X)$ is the subcomplex formed by all those currents that vanish on smooth null-$D$ test forms (i.e. on forms whose restriction to $D_\reg$ is zero).
The cohomology of $\calD(X,\log D)$ is isomorphic to the $\C$-valued singular cohomology of $X \setminus D$. 

To compute cohomology with integral coefficients or coefficients in a ring $A$, one can use the complex of relative integral currents $\calI(X,D,A)$ as introduced by Federer/Fleming \cite{FedererFleming}. An integral current is a current that together with his boundary is a limit of (integration currents over) singular chains. The complex of relative integral currents is the quotient complex $\calI(X)/\calI(D)$ (if necessary, tensored with $A$) of integral currents on $X$ by the integral currents with support\footnote{The support $\spt(T)$ of a current is the smallest closed subset $C$ such that $T|_{X \setminus C} = 0$. 
The currents $[\log]$ and $[\dlog]$ for example both have support $\P^1(\C)$, while the support of $[Z]$ is the closure of $Z$.
} in $D$.    

% Examples
The main examples of currents arise from integration. 
So any subvariety $Z \subset X$ and every locally integrable differential forms $w$ give rise to currents on $X$ by means of integration
\[
    [Z](\eta) = \int_{Z_\reg} \eta , \hspace{2cm} [w](\eta) = \int_X w \wedge \eta.
\]
By the same formula, every algebraic subvariety in $X \setminus D$ gives rise to a current in $\calD(X,\log D)$ and even in $\calI(X,D,\Z)$. This is the simple extension of the current $[Z_\reg]$ to $X$ (\cite{Lelong_IntegrationSurUnEnsembleAnalytiqueComplexe}, \cite{Herrera_OnTheExtensionOfCurrents}). As a log current, this is the same as the current of integration over the regular part of the closure $\overline{Z} \subset X$. In particular, the current is $d$-closed and integral of bidegree $(p,p)$, where $p$ is the codimension of $Z \subset X$. 

If $\omega$ is a smooth differential form on $X \setminus D$ with at most logarithmic poles along $D$, then $[\omega]$ is also a current in $\calD(X,\log D)$. With the above defined differential, the induced map from logarithmic forms, $\calA(X,\log D) \to \calD(X, \log D)$ turns out to be a quasi-isomorphism of complexes \cite[Thm 2.1.2]{King_LogComplexesOfCurrents}. If the Hodge filtration is defined to be the usual descending filtration with respect to the first variable: $F^p \calD(X) := \bigoplus_{r \geq p} \calD^{r,s}(X)$, the this map becomes a filtered quasi-isomorphism (i.e. induces a quasi-isomorphism on each graded piece).

Any current $T$ can be multiplied with a smooth (even locally integrable) differential form $w$ by $(T \wedge w)(\eta) = T(w \wedge \eta)$. Especially when restricted to $0$-forms, one gets a $\calA^0(X)$-module structure.

\paragraph*{Functoriality}
Currents are functorial in $X$: covariant functorial for proper and contravariant for submersions, where the induced mappings are defined dually through pullback of forms and "fiber integration".
This functoriality extends to currents on pairs (log currents and relative integral currents) with properties as described in \ref{appendix:requirements_on_D_complexes}.

\paragraph*{Exterior product} 
For two currents $S,T$ on $X$ resp $Y$ there exists an associative graded-commutative (with respect to dimension) external product $\boxtimes$ that satisfies the product rule with respect to the differential $d$. It is uniquely defined by 
\[
   (S \boxtimes T) (\eta_1 \boxtimes \eta_2) = (-1)^{|T| |\eta_1|}  \cdot S(\eta_1) \cdot T(\eta_2).
\]

\paragraph*{Intersection product}
There is also an internal (intersection-) product $\cap$ for currents, that extends both the wedge product of forms and the transversal intersection of submanifolds. As the general intersection of submanifolds, it is only partially defined. 
Even if it is defined, the intersection of two integral currents need not be integral again. However, after moving currents in their cohomology class, their intersection exists and, if both operands are integral, is integral again.
In any case, it is associative, graded-commutative with respect to degree and satisfies the Leibniz rule with respect to $d$. 

In some special cases, the intersection product takes a concrete form. For example the intersection product of two currents associated to properly intersecting (real) subvarieties is equal to the current associated to the intersection of them, 
\[
    [Z_1] \cap [Z_2] = [Z_1 \cap Z_2].
\]
Another important case is the intersection of a current with the current associated to a differential form $w$. Under some conditions, 
\[
    T \cap [w] = T \wedge w.
\]
This holds for example if $T$ is locally flat and $w$ is smooth. If $w$ is a form with logarithmic poles along some normal crossing divisor $D$, then the wedge product is a priori only a log current, i.e. only defined on test forms that vanish on $D_\reg$. 
\\
If $T = [Z]$ is the current of integration over a complex subvariety, then as log currents,
\[
    ([Z] \wedge \omega)(\eta)  = \int_{(Z \setminus D)_\reg} w \wedge \eta.
\]

\paragraph*{Properties of the intersection product}
We will use the axioms of the intersection product of currents given by Robert Hardt in \cite{HardtRealAnalyticIntersectionTheory}. 
%Hardt examined the intersection for a large class of currents, the "real analytic chains", and gave conditions for their intersection to exist.
%He furthermore gives a list of axioms that uniquely determine this intersection theory. 
%If one is content that the intersection product is only partially defined, as we do, then one gets along with properties for real locally flat current (not necessarily integral). 
%
While Hardt considerred a subclass of currents, the "real analytic chains", we will use his axioms on the whole class of currents, with the constraint that the product is only partially defined. 
We now state his axioms and some further properties of this intersection theory (on the complex manifold $X$).
The following statements have to be read as: If one side of the equation exists, then so does the other, and both sides are equal. 
Let $R,S,T$ be locally flat currents of degrees $r,s,t$. Then 
\begin{itemize}
\item $(\lambda S) \cap T = \lambda ( S \cap T)$ for any constant $\lambda$.   
\item $S \cap T = (-1)^{rs} T \cap S$.   
\item $(S \cap T)|_U = S|_U \cap T|_U$ for every open subset $U$.
\item $\varphi_* (S \cap T) = \varphi_* S \cap \varphi_* T$ for every orientation-preserving analytic isomorphism between analytic manifolds.   
\item $\Delta_* (S \cap T) = (S \boxtimes T) \cap \Delta_*[X]$ where $\Delta$ is the diagonal embedding.
\item $(R \cap S) \cap T = R \cap (S \cap T)$ if the currents intersect suitably good. 
\item If $Y$ is another complex manifold, $L$ a locally flat current in $X \times Y$ and $R$ a locally flat current in $X$ such that $\pr_X|_{\spt L}$ is proper and $L \cap (R \boxtimes [Y])$ exists. Then the intersection of $\pr_* L$ with $R$ exist and $\pr_* L \cap R = \pr_* (L \cap (R \boxtimes [Y]))$.
\item $[0] \cap [0] = [0]$ in $\R^0 = \{ 0 \}$.
\end{itemize}

Hardt proved further properties of his intersection theory, for example
\begin{itemize}
\item (Leibniz rule)  $d (S \cap T) = d S \cap T + (-1)^{s} S \cap d T$. 
\item (Compatibility of $\boxtimes$ and $\cap$) 
      $(S \cap T) \boxtimes (P \cap Q) = (-1)^{t p} (S \boxtimes P) \cap (T \boxtimes Q)$.
\item (Inverse mapping formula)
      If $b : X \to Y$ is an analytic mapping such that $b^* T$ is a locally flat current on $X$
      and if $S$ is a locally flat current on $X$ such that $b|_{\spt S}$ is proper, then 
      \[
           T \cap b_* S = b_* ( (b^* T) \cap S).         
      \]  
\end{itemize}

\paragraph{Intersection and Hodge filtration}
The intersection product of currents, whenever it exists, adds up the degrees of the involved currents. It is however not clear (or at least unknown to the author) whether intersection preserves the Hodge filtration, i.e. whether
\[
    F^p \calD(X) \cap F^q \calD(X) \subset F^{p+q} \calD(X).
\]
Recall that the Hodge filtration is the descending filtration $F^p \calD(X) := \bigoplus_{r \geq p} \calD^{r,s}(X)$. That is, a current lies in $F^p$ if and only if it vanishes on all test forms of bidegree $(i,j)$, $i>m-p$. Whenever a sequence of currents $S_y$ weakly converges to a current $S_0$, and if each $S_y$ vanishes on all test forms of some $(r,s)$ type, then $S_0$ also vanishes on all $(r,s)$ forms. 

To see that the intersection is additive in the Hodge filtration, we apply this to the intersection of currents. For technical reasons, we restrict ourselves to flat currents (which is no restriction, since every homology class has a locally flat representative). 
It suffices to work locally, so that we can assume that $X = \R^n$. The intersection of two currents $S,T$ (if it exist) is the unique current that, when identified with its image on the diagonal in $\R^n \times \R^n$, is equal to the slice  $\pm < S \boxtimes T , \xi, 0 > $. 
%By definition, the intersection product is the unique current which, when identififed with a current supported on the diagonal, is the weak limit (if it exists) of the slices $\pm < S \times T , \xi, y > $ for $y \to 0$. 
Now the exterior product is additive in the bidegree, and thus it suffices to show the general statement that a slice $<T,\xi,y>$ has bidegree $(r+m,s+m)$, if $T$ has bidegree $(r,s)$. This follows, since the slice is defined as the weak limit 
\[
    <T,\xi,0> := \lim_{\rho \to 0} T \wedge \xi^*(\mathbbm{1}_{B(0,\rho)}w)
\]
for $w$ the volume form. Since $w$ has pure type $(m,m)$, pullback preserves the bidegree, and $\wedge$ is a bi-additive operation, the result follows.

\bibliographystyle{plain}
\bibliography{Bibliography}

% Provide Author information
\par \vspace{2\baselineskip} \footnotesize
\textsc{Thomas Weißschuh, Institut für Mathematik, Johannes Gutenberg Universität Mainz, Germany.} \\
\textit{E-Mail address:} \textbf{weisssth@uni-mainz.de}

\end{document}